\documentclass[12pt,righttag]{amsart}
\usepackage{amscd}
\usepackage{amssymb}
\usepackage{subfigure}
\usepackage{graphicx}
\usepackage{verbatim}
\usepackage[all]{xy}
\usepackage{mathrsfs}
\usepackage{amsmath}
\usepackage{latexsym}
 \newtheorem{theorem}{Theorem}[section]
 \newtheorem{Def}[theorem]{Definition}
 \newtheorem{Prop}[theorem]{Proposition}
 \newtheorem{Lem}[theorem]{Lemma}
 \newtheorem{Cor}[theorem]{Corollary}
 \newtheorem{Rem}[theorem]{Remark}
  
  \newtheorem{Exa}[theorem]{Example}

\newcommand{\ba}{\begin{array}}
\newcommand{\ea}{\end{array}}
\newcommand{\beq}{\begin{equation}}
\newcommand{\eeq}{\end{equation}}

\newcommand{\set}[2]{\left\{ #1 : #2 \right\}}
\newcommand{\RR}{\mathbb{R}}

\newcommand{\ZZ}{\mathbb{Z}}
\newcommand{\NN}{\mathbb{N}}

 \numberwithin{equation}{section}
 
\begin{document}

\title{ Boundaries of disk-like self-affine tiles}


\author{King-Shun Leung} \address{ Department of Mathematics and Information Technology,
The Hong Kong Institute of Education, Hong Kong}
\email{ksleung@ied.edu.hk}

\author{Jun Jason Luo} \address{Department of Mathematics, Shantou University, Shantou, Guangdong 515063, China }
\email{luojun2011@yahoo.com.cn}

\thanks{The research is supported by STU Scientific Research Foundation for Talents (no. NTF12016).}

\date{\today}

\begin{abstract}
Let $T:= T(A, {\mathcal D})$ be a  disk-like self-affine tile generated by an integral expanding matrix $A$ and a consecutive collinear digit set ${\mathcal D}$, and let $f(x)=x^{2}+px+q$ be the characteristic polynomial of $A$. In the paper, we identify the boundary $\partial T$ with a sofic system by constructing a neighbor graph and  derive equivalent conditions for the pair $(A,{\mathcal D})$ to be a number system. Moreover, by using the graph-directed construction and a device of  pseudo-norm $\omega$, we find the generalized Hausdorff dimension $\dim_H^{\omega} (\partial T)=2\log \rho(M)/\log |q|$ where $\rho(M)$ is the spectral radius of certain contact matrix $M$. Especially, when $A$ is a similarity, we obtain the standard Hausdorff dimension $\dim_H (\partial T)=2\log \rho/\log |q|$ where $\rho$ is the largest positive zero of the cubic polynomial $x^{3}-(|p|-1)x^{2}-(|q|-|p|)x-|q|$, which is simpler than the known result.
\end{abstract}

\maketitle

\begin{section} {\bf Introduction}
Let $M_n(\Bbb Z)$ denote the set of $n\times n$ matrices with
entries in $\Bbb Z$ and let $A\in M_n(\Bbb Z)$ be expanding (i.e.,
all eigenvalues of $A$ have moduli $> 1$). Assume $|\det(A)|= |q|$,
and ${\mathcal D}=\{0,d_1,\dots, d_{|q|-1}\}\subset {\Bbb Z}^n$ with
$|q|$ distinct vectors. We call ${\mathcal D}$ a {\it digit set} and
$(A, {\mathcal D})$ a {\it self-affine pair}. It is well-known that
there exists a unique self-affine set $T:= T(A, {\mathcal D})$
\cite{[LW1]}  satisfying
$$T=A^{-1}(T+{\mathcal D})=\left\{\sum_{i=1}^{\infty}A^{-i}d_{j_i}: d_{j_i}\in
{\mathcal{D}}\right\}.$$ If $T$ has non-void interior, then there
exists a subset ${\mathcal J}\subset {\Bbb Z}^n$ such that
$$T+{\mathcal J}={\Bbb R}^n\quad \text{and}\quad (T+t)^\circ\cap
(T+t')^\circ=\emptyset, \ t\ne t',\ t,t'\in {\mathcal J},$$ thus $T$
is called a {\it self-affine tile} and ${\mathcal J}$ a tiling set.
$T+{\mathcal J}$ is called a tiling of ${\Bbb R}^n$, and a lattice
tiling if ${\mathcal J}$ is a lattice \cite{[LW3]}.

The topological properties of self-affine tiles and their boundaries, such as
connectedness, local connectedness or disk-likeness (i.e., homeomorphic to the closed
unit disk),  have attracted a lot of interest.  A systematical study
on the connectedness of self-affine tiles was due to  Kirat and Lau \cite{[KL]}, they mainly
concerned  a class of  tiles $T(A, {\mathcal D})$ generated by the {\it
consecutive collinear (CC)} digit sets ${\mathcal D}:={\mathcal
D}(v,|q|)=\{0,1,\dots,|q|-1\}v, v\in {\mathbb Z}^n\setminus\{0\}$
via the algebraic property of the characteristic polynomial of the
matrix $A$. More general cases on non-consecutive collinear or
non-linear digit sets were considered by \cite{[K]}, \cite{[DL]}, \cite{[LLu]}.

The question of disk-likeness was first investigated by Bandt and
Gelbrich \cite{[BG]} for self-affine tiles in ${\Bbb R}^2$ with
$|\det(A)|=2$ or $3$.  They observed that the characteristic
polynomial of $A\in M_2(\Bbb Z)$ is of the form:
$$f(x)=x^2+px+q, ~\text{with}~ |p|\leq q, ~\text{if}~ q\geq
2;\quad |p|\leq |q+2|, ~\text{if}~ q\leq -2.$$  By studying the
neighborhood structure of $T$, Bandt and Wang \cite{[BW]} proved
that a tile $T$ with no more than six neighbors is disk-like if and
only if $T$ is connected.  A translation of the tile $T+\ell,\
\ell\in {\mathcal J}$ is called  a {\it neighbor} of $T$ if $T\cap
(T+\ell)\ne \emptyset$.  Making use of this criterion, Leung and Lau
\cite{[LL]} then gave a complete characterization of the
disk-likeness of self-affine tiles with CC digit sets.  Gmainer and Thuswaldner \cite{[GM]} considered the disk-likeness of tiles with non-collinear digit sets arising from polyominoes, and Kirat \cite{[K]} proposed necessary and sufficient conditions for such tiles to be disk-like in general. By using the neighbor map technique, Bandt and Mesing \cite{[BM]} constructed a kind of finite type self-affine tiles and discussed their disk-likeness as well.

\begin{theorem}(\cite{[LL]})\label{disk-likeness thm}
Let $A\in M_2(\Bbb Z)$ be an expanding matrix with characteristic
polynomial $f(x)=x^2+px +q$. Then for any CC digit set ${\mathcal
D}(v, |q|)$ in ${\Bbb Z}^2$ such that $v, Av$ are linearly
independent, $T$ is a disk-like tile if and only if $2|p|\leq
|q+2|$.

Moreover, when $p=0$, $T$ is a square tile; when $p\ne 0$, $T$ is a
hexagonal tile.
\end{theorem}

The boundary of a self-affine tile has more complicated geometric structure than the tile itself, hence it is also of considerable interest. The dimension of the boundary of a self-similar tile (where the expanding matrix $A$ is a similarity) has been studied extensively in the literature. Strichartz and Wang \cite{[SW]} described the boundary set as a graph-directed set and gave an algorithm for finding the dimension of the boundary,  various other methods can be founded in \cite{[DKV]}, \cite{[V]}, \cite{[HLR]}, \cite{[LN]}.

Recently,  Akiyama and Loridant (\cite{[AL]}, \cite{[AL2]}) provided
a new method to parameterize the boundary set and reproved Theorem \ref{disk-likeness thm} by showing that the
boundary of $T$ is a simple closed curve. In the present paper, we go further to explore the structure of the boundary of the $T$
defined in Theorem \ref{disk-likeness thm}. For convenience, we call
such $T$ a {\it CC tile}. If it is also disk-like, we call it a {\it
disk-like CC tile}.

First we establish a neighbor graph of $T$ such that the boundary $\partial T$ is identified as the union of all one-sided infinite paths of this graph. Hence $\partial T$ determines a sofic system \cite{[Fi]}. The neighbor graph technique is classical in the study of tiling theory (\cite{[BG]}, \cite{[BM]}).  However, it will be shown that we use the technique here from a different aspect.  As self-affine tiles can be studied in the context of number systems \cite{[MTT]}, it is worth studying the conditions for the self-affine pair $(A,{\mathcal D})$ to be a number system. We give the answer when $T(A,{\mathcal D})$ is disk-like.

\begin{theorem}\label{thm2}
Let $T=T(A,{\mathcal D})$ be a disk-like CC tile. Then the following
are equivalent:

(i) $(A,{\mathcal D})$ is a number system.

(ii) $0\in T^{\circ}$.

(iii) $f(x)=x^{2}+px+q$ with $-1\leq p$ and $q\geq 2$.

(iv) For all neighbors $T+\ell$, \
$\ell=\sum_{i=0}^{k}a_{i}A^{i}v\in {\mathcal D}_{A,k+1}$ for some
$k\in\ZZ$ with $a_k=1$ and $a_{i}\in D$ where $0\leq i<k$.
\end{theorem}

In \cite{[SW]},  Strichartz and Wang applied the graph-directed iterated function system (GIFS) to represent the boundary of a self-affine tile, but they were not sure whether the GIFS satisfies the open set condition or not. Our second aim is to give a positive answer for the disk-like CC tile and estimate the generalized Hausdorff dimension ($\dim_H^{\omega}$) of the boundary by using a pseudo-norm $\omega$ (\cite{[HL]},\cite{[LY]}) instead of Euclidean norm.

\begin{theorem}\label{thm3}
The generalized Hausdorff dimension of the boundary of disk-like CC tile $T$ is $$\dim_H^{\omega}(\partial T)=\frac{2\log\rho(M)}{\log{|q|}}$$ where $\rho(M)$ denotes the spectral radius of certain contact matrix $M$,  and the corresponding measure is positive and finite.
\end{theorem}

 When $A$ is a similarity, we can improve  the well-known  Hausdorff dimension formula of the boundary in the following simpler way.

\begin{theorem}\label{mainthm}
Let $A\in M_{2}(\ZZ)$ be an expanding similarity with characteristic
polynomial $f(x)=x^{2}+px+q$ and $T=T(A,{\mathcal D})$ be a
disk-like CC tile. Then
$$\dim_H \big(\partial
T\big)=\frac{2\log\rho}{\log |q|}$$ where $\rho$ is the largest
positive zero of the cubic polynomial
$x^{3}-(|p|-1)x^{2}-(|q|-|p|)x-|q|.$
\end{theorem}

The rest of the paper is organized as follows: In Section 2, we identify  $\partial T$ with a  sofic system by constructing a neighbor graph and prove Theorem \ref{thm2}. In Section 3, we consider $\partial T$ as a graph-directed set and prove Theorems \ref{thm3} and \ref{mainthm}. Finally all neighbor graphs,
graph-directed sets and contact matrices corresponding to different characteristic polynomials $f(x)$ are  listed in  Appendices A, B and C for easy reference.
\end{section}

\bigskip

\begin{section}{\bf Sofic System and number system}

We first introduce some terminology of symbolic dynamics from
\cite{[LM]}. Let $\mathcal{G}=\mathcal{G(V,E)}$ be a directed graph
where $\mathcal{V}$ is the set of vertices and $\mathcal{E}$ the set
of edges.  Let $\mathcal{A}$ be a finite set (called {\it
alphabet}). If there exists a mapping (called {\it labeling})
$\mathcal{L}:\mathcal{E}\rightarrow\mathcal{A}$, then the ordered
pair $\mathbf{G}=(\mathcal{G},\mathcal{L})$ is called a {\it labeled
directed graph}. All the infinite paths $\xi=e_{1}e_{2}e_{3}\ldots$
on $\mathcal{G}$ constitute the so-called {\it edge shift}
$\mathbf{X}_{\mathcal{G}}$.  Define the {\it label of the path
}$\xi$ by
\begin{equation*}
\mathcal{L}_{\infty}(\xi):=
\mathcal{L}(e_{1})\mathcal{L}(e_{2})\mathcal{L}(e_{3})\ldots\in\mathcal{A}^{\NN}.
\end{equation*}
Here $\mathcal{A}^{\Bbb N}$ is called the {\it full shift} of
$\mathcal{A}$.  The set of all such labels is denoted by
$$
\mathbf{X}_{\mathbf{G}}=\set{x\in\mathcal{A}^{\NN}}{x=\mathcal{L}_{\infty}(\xi)\;\text{for
some}\;\xi\in\mathbf{X}_{\mathcal{G}}}.
$$
Any subset of $\mathcal{A}^{\NN}$ which can be defined by a labeled
directed graph as above, is called a {\it sofic shift} or {\it sofic
system} (\cite{[Fi]}, \cite{[LM]}).  Weiss \cite{[We]} coined the
term {\it sofic} which is derived from the Hebrew word for {\it
finite} \cite{[LM]}.

Let $D=\{0,1,\dots, |q|-1\}$ and the difference set  $\Delta D :=D-D$,  then the CC digit set ${\mathcal D}= Dv$ and $\Delta{\mathcal
D} := {\mathcal D}- {\mathcal D}= \Delta D v$. Without loss of generality, we assume the digit set ${\mathcal D}$ is primitive, i.e., the lattice ${\mathcal J}$ generated by ${\mathcal D}$ and $A{\mathcal D}$ in ${\mathbb Z}^2$ is equal to ${\mathbb Z}^2$. For otherwise, there exists an invertible $B\in M_2({\mathbb Z})$ such that $\tilde{\mathcal D}=B^{-1}{\mathcal D}\subset {\mathbb Z}^2$ is primitive and $T(A,{\mathcal D})=BT(\tilde{A}, \tilde{\mathcal D})$ where $\tilde{A}=B^{-1}AB\in M_2({\mathbb Z})$ \cite{[LW2]} and we can consider $\tilde{A}, \tilde{\mathcal D}$ instead. Hence we set ${\Bbb Z}^2 =\{\gamma v + \delta Av: \gamma, \delta \in {\Bbb Z}\}$. It is easy to see that $T+\ell$ where $\ell \in {\mathbb Z}^2$ is a neighbor of $T$ if and only if $\ell\in T-T$.  More precisely, $\ell$ can be expressed as
$$\ell=  \sum_{i=1}^{\infty}b_i A^{-i}v \in T-T, \quad  b_i\in \Delta D.$$

The following is a neighbor-generating formula which plays a key role in constructing the labeled directed graph for the boundary.

\begin{Lem}(\cite{[LL]}) \label{neighbor generator}
Suppose $T+\ell$ is a neighbor of $T$ with $\ell=\gamma v+\delta Av=\sum_{i=1}^{\infty}b_iA^{-i}v$, then we get  another neighbor
$T+\ell'$  satisfying $\ell'=A\ell-b_1v=\gamma' v+\delta' Av$ with $\gamma'=-(q\delta+b_1)$ and $\delta'=\gamma-p\delta$.

Inductively, we can construct a sequence of neighbors: $\{T+\ell_n\}_{n=0}^{\infty}$ where $\ell_0=\ell$ and $\ell_{n+1}= A\ell_n-b_{n+1}v$.
\end{Lem}

 Let $T$ be a disk-like CC tile and $T_\ell=T\cap(T+\ell)$ for any $\ell\in {\mathbb Z}^2$. Let $\mathcal{V}=\{\ell\in {\mathbb Z}^2: \ell\ne 0 \ \text{and} \ T\cap T_\ell\ne\emptyset\}$. Then the boundary of $T$ can be written as
\begin{equation}\label{eq:bdryT}
\partial T =\bigcup_{\ell\in\mathcal{V}}T_{\ell}.
\end{equation}

Define an edge set  ${\mathcal E}:=\{e=(\ell,\ell'):\  \ell,\
\ell'\in { \mathcal V}\ \ \text{and}\ \ \ell'=A\ell-b_{1}v
 \ \text{for some} \
b_{1} \in \Delta D\}$ and a labeling
$\mathcal{L}:\mathcal{E}\rightarrow\mathcal{A}$ by
$\mathcal{L}(e)=b_{1}$ where ${\mathcal A}=\Delta D$. Then by the
definition above,  $\mathbf{G}=(\mathcal{G},\mathcal{L})$ is a
labeled directed graph and it determines a sofic shift. We call
$\mathbf{G}$ the \textit{neighbor graph} of $T$.

\begin{Prop}\label{prop:labdigraph}
Let $\mathbf{G}$ be the neighbor graph of a CC disk-like tile $T$.
If $x=\sum_{i=1}^{\infty}a_{i}A^{-i}v=\ell
+\sum_{i=1}^{\infty}a'_{i}A^{-i}v\in T_{\ell}$ where $a_{i},
a'_{i}\in D$, then $\{b_{i} :=a_{i}-a'_{i}\}_{i=1}^{\infty}$ is the
sequence of labeling of the edges of an infinite path starting at
$\ell$ (or simply called a label sequence starting at $\ell$).
Conversely, any label sequence $\{b_{i}\}_{i=1}^{\infty}$ (with
$b_{i}\in \Delta D$) starting at $\ell$ defines a set
\begin{equation*}
\big\{x:x=\sum_{i=1}^{\infty}a_{i}A^{-i}v=\ell+\sum_{i=1}^{\infty}a'_{i}A^{-i}v,
\  a_{i}-a'_{i}=b_{i}, \  a_{i}, a'_{i}\in D \ \text{for} \
i=1,2,\dots\big\}
\end{equation*}
of boundary points of $T$.
\end{Prop}

\begin{proof}
Since $\ell=\sum_{i=1}^{\infty}b_{i}A^{-i}v$ with $b_i=a_i-a_i'$, by
Lemma \ref{neighbor generator}, we have a sequence of neighbors
$\{T+\ell_n\}_{n=0}^{\infty}$ where $\ell_{0}=\ell$ and
$\ell_{n+1}=A\ell_n-b_{n+1}v$, hence $\{b_i\}_{i=1}^{\infty}$ is a
label sequence starting at $\ell$ by the definition.

Conversely, if $\ell=\sum_{i=1}^{\infty}b_{i}A^{-i}v$ where
$b_{i}\in\Delta D$, then $b_i=a_{i}-a_i'$ for $a_{i}, a'_{i}\in D$
and $\ell=\sum_{i=1}^{\infty}(a_{i}-a'_{i})A^{-i}v$.   It follows
that
\begin{equation}\label{eq:bdry-pt}
x=\sum_{i=1}^{\infty}a_{i}A^{-i}v=\ell+\sum_{i=1}^{\infty}a'_{i}A^{-i}v\in
T\cap(T+\ell)=T_{\ell}.
\end{equation}
\end{proof}

We can verify whether the origin $0$ is a boundary point of $T$  in the following way.

\begin{Cor}\label{cor:origin}
 $0\in \partial T$ if and only if there
exists an infinite path in $\mathbf{G}$ with all edge labels either
non-positive or  non-negative.
\end{Cor}

\begin{proof}
Suppose $0\in T\cap(T+\ell)$ for some neighbor $T+\ell$.  Putting
$a_{i}=0$ for all $i$ into (\ref{eq:bdry-pt}), we have
$$
\ell=\sum_{i=1}^{\infty}(-a'_{i})A^{-i}v.
$$
Since $a_i'\in D$, the label sequence
$\{b_{i}=-a'_{i}\}_{i=1}^{\infty}$ starting at $\ell$ has all labels
non-positive.  Similarly $\{b'_{i}=a_{i}\}_{i=1}^{\infty}$
 is a sequence starting at $-\ell$ with all labels non-negative. By reversing the argument, we can prove the converse.
\end{proof}

In fact,  we can determine the neighbor graph $\mathbf{G}$ for any
disk-like CC tile $T$. Let us take the case of
 $f(x)=x^{2}+px+q,\;p,q\geq 2$ (excluding $p=q=2$) as an example. By
Theorem \ref{disk-likeness thm},  $T$ is a hexagonal tile with six
neighbors \cite{[LL]} and
\begin{equation}\label{six neighbors}
\mathcal{V}=\{\pm v,\;\pm \big(Av+(p-1)v\big),\;\pm(Av+pv)\}.
\end{equation}
In view of the definition of $\mathcal{E}$, if $\ell=v$ we take
$b_{1}=-p$ and $\ell'=Av+pv$ or $b_{1}=-(p-1)$ and
$\ell'=Av+(p-1)v$;  if $\ell=Av+pv$,  using $f(A)v=0$, we have
$b_{1}=-(q-1)$ and $\ell'=-v$. Proceeding similarly with all $\ell$,
we obtain Table \ref{nbtb:x^2+px+q}. Then we establish the neighbor
graph (Figure \ref{fig:0g}). The neighbor graphs corresponding to
other $f(x)$ are given in Appendix A.

\begin{table}[h]
\centering
\begin{tabular}{|c|c|c|}
\hline $\ell$ & $b_{1}$ & $\ell'$\\
\hline
$v$ & $-(p-1)$ & $Av+(p-1)v$\\
& $-p$ & $Av+pv$\\
\hline
$Av+(p-1)v$ & $-(q-p)$ & $-Av-pv$\\
& $-(q-p+1)$ & $-Av-(p-1)v$\\
\hline
$Av+pv$ & $-(q-1)$ & $-v$\\
\hline
$-v$ & $p-1$ & $-Av-(p+1)v$\\
& $p$ & $-Av-pv$\\
\hline
$-Av-(p-1)v$ & $q-p$ & $Av+pv$\\
& $q-p+1$ & $Av+(p-1)v$\\
\hline
$-Av-pv$   & $q-1$ & $v$\\
\hline
\end{tabular}
\vspace{0.2cm} \caption{Relation among all neighbors of $T$
associated with $f(x)=x^{2}+px+q,\;p,q\geq 2,\;2p\leq
q+2$\;(excluding\;$p=q=2$).} \label{nbtb:x^2+px+q}
\end{table}

\begin{figure}[h]
  \[
  \begin{xy}
    0;<1cm,0cm>: (0,0)*=<2cm,1cm>{v}*\frm{-}="a"
    ,(0,-3)*=<3cm,1cm>{Av+(p-1)v}*\frm{-}="b"
    ,(0,-6)*=<2cm,1cm>{Av+pv}*\frm{-}="c"
    ,(6,0)*=<2cm,1cm>{-v}*\frm{-}="d"
    ,(6,-3)*=<3cm,1cm>{-Av-(p-1)v}*\frm{-}="e"
    ,(6,-6)*=<2cm,1cm>{-Av-pv}*\frm{-}="f"
    ,"a"+(0,-0.5)
    ;"b"+(0,0.5)**\dir{-}?*^!/25pt/{-(p-1)}?>*\dir{>}
    ,"b"+(1.5,-0.5)
    ;"f"+(-1,0.5)**\dir{-}?(0.1)*^!/20pt/{-(q-p)}?>*\dir{>}
    ,"e"+(-1.5,-0.5)
    ;"c"+(1,0.5)**\dir{-}?(0.1)*_!/20pt/{q-p}?>*\dir{>}
    ,"b"+(1.5,0.3)
    ;"e"+(-1.5,0.3)**\dir{-}?*_!/10pt/{-(q-p+1)}?>*\dir{>}
    ,"e"+(-1.5,-0.3)
    ;"b"+(1.5,-0.3)**\dir{-}?*_!/10pt/{q-p+1}?>*\dir{>}
    ,"a"+(-1,0)
    ;"c"+(-1,0)**\crv{(-4,-3)}?*^!/10pt/{-p}?>*\dir{>}
    ,"d"+(0,-0.5)
    ;"e"+(0,0.5)**\dir{-}?*_!/25pt/{p-1}?>*\dir{>}
    ,"d"+(1,0)
    ;"f"+(1,0)**\crv{(10,-3)}?*_!/10pt/{p}?>*\dir{>}
    ,"c"+(1,-0.5)
    ;"d"+(1,0.3)**\crv{"f"+(0,-2)&"f"+(4,0)&"d"+(4,0)}?(0.3)*^!/15pt/{-(q-1)}?>*\dir{>}
    ,"f"+(-1,-0.5)
    ;"a"+(-1,0.3)**\crv{"c"+(0,-2)&"c"+(-4,0)&"a"+(-4,0)}?(0.3)*_!/15pt/{q-1}?>*\dir{>}
  \end{xy}
  \]
  \vspace{0.2cm}
  \caption{The neighbor graph of $T$ associated with $f(x)=x^2+px+q,\;p,q\geq 2,\;2p\leq q+2$, (excluding $q=p=2$).}
  \label{fig:0g}
\end{figure}

\medskip

Following \cite{[LW1]}, we let ${\mathcal
D}_{A,k}=\{\sum_{i=0}^{k-1}a_iA^i v: a_i\in D\}$, $\Delta{\mathcal
D}_{A,k}={\mathcal D}_{A,k}-{\mathcal
D}_{A,k}=\{\sum_{i=0}^{k-1}b_iA^i v: b_i\in \Delta D\}$ and
${\mathcal D}_{A,\infty}=\cup_{k=1}^{\infty}{\mathcal D}_{A,k}$.

\begin{Prop}\label{prop:nbcoeff}
Let $T$ be a disk-like CC tile and $T+\ell$ a neighbor. Then
$\ell=\sum_{i=0}^k b_{i}A^{i}v\in\Delta{\mathcal D}_{A, k+1}$ for
some $k\in \ZZ$ with $b_k\in \{-1,1\}$, $b_{i}\in\Delta D$ for
$0\leq i<k$.  When $f(x)=x^{2}\pm 2x +2$, $k=3$; and $k=1$
otherwise.
\end{Prop}

\begin{proof}
It follows from (\ref{six neighbors}) that $\ell\in\Delta {\mathcal
D}_{A,2}$ excluding the case of  $f(x)=x^{2}\pm 2x+2$. For
$f(x)=x^{2}\pm 2x+2$, we have $Av \pm 2v =A^{3}v \pm
A^{2}v+Av\in\Delta{\mathcal D}_{A,4}$ by using $(A \mp I)f(A)=0$.
\end{proof}

A more desirable property is that  any $\ell\in {\Bbb Z}^2$ can be
expressed as $\ell=\sum_{i=0}^{k}a_{i}A^{i}v\in {\mathcal
D}_{A,k+1}$ (instead of $\Delta {\mathcal D}_{A,k+1}$)  for some
$k\in\ZZ$ with $a_k=1$ and $a_{i}\in D$ where $ 0\leq i<k$. But this
is not always the case. This property is closely related to  a {\it
number system} defined below (see also \cite{[MTT]}).

\begin{Def}\label{def:num-sys}
Let $A\in M_{2}(\Bbb Z)$ be expanding and  ${\mathcal D}$ be a CC digit set. The self-affine pair $(A,{\mathcal D})$ is
said to be a number system if for any $\ell\in \ZZ^{2}$, it has a
unique representation $\ell=\sum_{i=0}^{k}A^{i}v_{i}'$ with
$v_{i}'\in {\mathcal D}$.
\end{Def}

For convenience, we sometimes  write a point of the form
$x=\sum_{i=1}^{\infty}a_iA^{-i}v \in {\mathbb R}^2$ as {\it radix
expansion}: $0.a_1a_2a_3\ldots$  An overbar denotes repeating digits
as in $0.12\overline{301}=0.12301301301\ldots$ Likewise,
$a_{-2}a_{-1}a_0.a_1a_2a_3\ldots$ represents a point
$a_{-2}A^2v+a_{-1}Av+a_0v+\sum_{i=1}^{\infty}a_iA^{-i}v.$  Note that
shifting a radix place to the left means multiplying $A$ to $x$.
When $x$ is on the boundary of $T$, the radix expansion of $x$ is
not  unique.  Now we give some equivalent conditions for the
self-affine pair $(A,{\mathcal D})$ to be a number system.

\begin{theorem}\label{thm:numbsys}
Let $T=T(A,{\mathcal D})$ be a disk-like CC tile. Then the following
are equivalent:

 (i) $(A,{\mathcal D})$ is a number system.

(ii) $0\in T^{\circ}$.

(iii) $f(x)=x^{2}+px+q$ with $-1\leq p$ and $q\geq 2$.

(iv) For all neighbors $T+\ell$, \
$\ell=\sum_{i=0}^{k}a_{i}A^{i}v\in {\mathcal D}_{A,k+1}$ for some
$k\in\ZZ$ with $a_k=1$ and $a_{i}\in D$ where $0\leq i<k$.
\end{theorem}

\begin{proof}
(i)$\Rightarrow$(ii)\quad Suppose $0\notin T^{\circ}$. Then $0\in
T\cap(T+\ell)$ for some $\ell\in\ZZ^{2}\setminus\{0\}$. Since
$(A,{\mathcal D})$ is a number system,
$\ell=\sum_{i=-k}^{0}a_{i}A^{-i}v$ with $a_{i}\in D$ and $a_{-k}>0$.
Hence $0=a_{-k}a_{-(k-1)}\ldots a_{-1}a_{0}.a_{1}a_{2}a_{3}\ldots.$
Shifting the radix point $k$ places to the left, we get
$0=a_{-k}.a_{-(k-1)}\ldots a_{-1}a_{0}a_{1}a_{2}a_{3}\ldots.$ That
means $T+a_{-k}v$ is a neighbor of $T$.  By
Proposition~\ref{prop:nbcoeff}, $a_{-k}=1$. Hence  $0$ corresponds
to an infinite path starting at $v$ with non-positive labels
$b_{i}=-a_{i}$. But by checking all the neighbor graphs in Appendix
A, we find no such path.

(ii)$\Rightarrow$(i)\quad It suffices to show that $\ZZ^{2}\subset
{\mathcal D}_{A,\infty}$. By the lattice tiling property,  $0$ is
the only lattice point in $T$, i.e., $\ZZ^{2}\cap T=\{0\}$. It
follows that $\ZZ^{2}\cap A^{n}T=\sum_{i=0}^{n-1}A^{i}{\mathcal
D}={\mathcal D}_{A,n}$ for $n\geq 1$. If $\ell\in \ZZ^{2}$, there
exists a large integer $n$ such that $\ell\in A^{n}T$ as $0\in
T^{\circ}$, then $\ell\in{\mathcal D}_{A,n}\subset {\mathcal
D}_{A,\infty}$.

(ii)$\Leftrightarrow$(iii)\quad By inspecting all neighbor graphs in
Appendix A, we find that in each graph corresponding to
$f(x)=x^{2}+px+q$ with $-1\leq p$ and $q\geq 2$, there exists no
infinite path with edge labels either all non-positive or all
non-negative, hence $0\in T^\circ$ by Corollary \ref{cor:origin}. In
every other case, there always exists such a path. All these paths
are listed in Table~\ref{tab:path}.

{\small
\begin{table}[htbp]
\centering
{\renewcommand{\arraystretch}
{1.5}
\renewcommand{\tabcolsep}{0.2cm}
\begin{tabular}{|c|c|c|}
\hline
$f(x)$&Neighbor&Path\\
\hline
$x^{2}-2x+2$ &$Av-v$&$ \overline{(-1)}$\\
&$-Av+v$ & $\overline{1}$\\
\hline
$x^{2}-px-q$&$v$&$\overline{p(q-1)}$\\
&$-v$&$\overline{(-p)[-(q-1)]}$\\
\hline
$x^{2}+px-q$&$Av+(p+1)v$&$\overline{(q-p-1)}$\\
&$-Av-(p+1)v$&$\overline{[-(q-p-1)]}$\\
\hline
$x^{2}-q$&$v$&$\overline{0(q-1)}$\\
&$-v$&$\overline{0[-(q-1)]}$\\
&$Av+v$&$\overline{(q-1)}$\\
&$-Av-v$&$\overline{[-(q-1)]}$\\
\hline
$x^{2}-px+q$&$Av-(p-1)v$&$\overline{[-(q-p+1)]}$\\
&$-Av+(p-1)v$&$\overline{(q-p+1)}$\\
\hline
\end{tabular}}
\vspace{0.2cm}
 \caption{Infinite paths representing a boundary point
$0$.}\label{tab:path}
\end{table}}

(iii)$\Rightarrow$(iv)\quad  Let $f(x)$ be one of the cases:
$x^{2}+q$, $x^{2}+x+q$, $x^{2}+px+q\;(p\geq 2
,\;\text{excluding}\;p=q=2)$, $x^{2}+2x+2$, $x^{2}-x+q$, where
$p\geq 0$ and $q\geq 2$.  In each case, we can rewrite their
neighbors as the desired form in (iv).  By using  $0=f(A)v$,\
$0=(A-I)f(A)v$,\ $0=(A+I)f(A)v$, we have

Case (1) $f(x)=x^{2}+q$. \quad $Av-v  = A^{2}v+Av+(q-1)v,\ -v =
A^{2}v+(q-1)v,\ -Av = A^{3}v+(q-1)Av,\ -Av+v  = A^{3}v+(q-1)Av+v,\
-Av-v  = A^{3}v+A^{2}v+(q-1)Av+(q-1)v.$

Case (2) $f(x)=x^{2}+x+q$. \quad $-v   =  A^{2}v+Av+(q-1)v,\ -Av =
A^{3}v+A^{2}v+(q-1)Av,\ -Av-v  = A^{2}v+(q-1)v.$

Case (3) $f(x)=x^{2}+px+q\;(p\geq 2)$. \quad $-v   =
A^{2}v+pAv+(q-1)v,\
 -Av-(p-1)v  = A^{2}v+(p-1)Av+(q-p+1)v,\ -Av-pv=A^{2}v+(p-1)Av+(q-p)v.$

Case (4) $f(x)=x^{2}+2x+2$. \quad $Av+2v  = A^{3}v+A^{2}v+Av,\ -v =
A^{4}v +A^{3}v+A^{2}v+v, \ -Av-v  =  A^{2}v+Av+v,\ -Av-2v  =
A^{2}v+Av.$

Case (5) $f(x)=x^{2}-x+q$. \quad $Av-v  =  A^{2}v+(q-1)v,\ -v =
A^{3}v+(q-1)Av+(q-1)v,\  -Av  = A^{4}v+(q-1)A^{2}v+(q-1)Av,\
-Av+v=A^{4}v+(q-1)A^{2}v+(q-1)Av+v.$

(iv)$\Rightarrow$(ii)\quad Suppose $0\notin T^{\circ}$. By the same
argument as in the proof of (i)$\Rightarrow$(ii) above, there should
be an infinite path in the neighbor graph starting at $v$ with edge
labels all non-positive. But we find no such path by inspecting all
the neighbor graphs in Appendix A.
\end{proof}

\begin{Rem}
Gilbert \cite{[Gi]} obtained some related results in the context of
quadratic number fields.  We conjecture that Theorem
\ref{thm:numbsys} can be extended to non-disk-like tiles.
\end{Rem}

\end{section}

\bigskip
\begin{section}{\bf Dimension of the boundary of $T$}

For a directed graph $\mathcal{G}=\mathcal{G(V,E)}$ where
$\mathcal{V}=\{v_{1},v_{2},\dots,v_{m}\}$, we  write
$\mathcal{E}_{i,j}$ for the set of edges from vertex $v_{i}$ to
vertex $v_{j}$, and we add a contraction mapping
$F_{e}:\RR^{2}\rightarrow\RR^{2}$ for each edge $e\in\mathcal{E}$.
Then the family of contractions $\{F_e: e\in{\mathcal E}\}$ is
called a {\it graph-directed iterated function system
(GIFS)} and there exists a unique family of non-empty compact subsets $E_{1},\dots,E_{m}$ of
$\RR^{2}$ (\cite{[Fa2]}, \cite{[MW]}) such that
\begin{equation}\label{eq:GDsets}
E_{i}=\bigcup_{j=1}^{m}\bigcup_{e\in\mathcal{E}_{i,j}}F_{e}(E_{j}).
\end{equation}
We call $E:=\bigcup_{i=1}^mE_i$ a {\it graph-directed set}.  Define $M=(M_{ij})_{1\leq i,j\leq m}$ as the {\it contact matrix} \cite{[GH]} of $\mathcal{G}$ with $M_{ij}=\#\mathcal{E}_{i,j}$ counting the number of edges from $v_i$ to $v_j$.

The GIFS $\{F_e: e\in{\mathcal E}\}$ is said to
satisfy the {\it open set condition (OSC)} if there exist a family
 of open sets $\{O_{1}, \dots,O_{m}\}$  such that
\begin{equation}\label{eq:GDOSC}
O_{i}\supset\bigcup_{j=1}^{m}\bigcup_{e\in\mathcal{E}_{i,j}}F_{e}(O_{j})\quad\text{for}\;
i=1,2,\dots,m
\end{equation} with  disjoint unions, i.e., $F_{e}(O_{j})\cap F_{e'}(O_{j'})=\emptyset$
whenever $(e,j)\neq (e',j')$. With this OSC, we then can compute the dimension of the graph-directed set.

In this section, we first identify the boundary of $T$ with a graph-directed set by making use of the well-known method (\cite{[SW]}, \cite{[HLR]}), then calculate its dimension in the self-affine case and the self-similar case, respectively.

\begin{Prop}\label{prop:union}
Let $\ell=\gamma v+\delta Av$, $\ell'=\gamma' v+\delta' Av
\in\mathcal{V}$ such that $\ell'=A\ell-b_{1}v$ for some
$b_{1}\in\Delta D$, then
\begin{equation*}
A^{-1}(T_{\ell'}+jv)\subset T_{\ell}\quad\text{for all}\quad j\in
I_{b_{1}}:= \left\{\begin{array}{ll}
\{b_{1},b_{1}+1,\dots,q-1\}&\text{if}\quad b_{1}\geq 0;\\
\{0,1,\dots,q-1+b_{1}\}&\text{if}\quad b_{1}<0.
\end{array}\right.
\end{equation*}
Moreover,
\begin{equation*}
T_{\ell}=\bigcup_{\ell'\in B_{\ell}}\bigcup_{j\in
I_{b_{1}}}A^{-1}(T_{\ell'}+jv)
\end{equation*} where $B_{\ell}:=\{\ell''\in\mathcal{V}: \ell''=A\ell-b_1'v \ \text{for
some}\ b_1'\in\Delta D\}$. Hence the boundary $\partial T=\bigcup_{\ell\in {\mathcal V}}T_\ell$ is a graph-directed set.
\end{Prop}

\begin{proof}
When $b_{1}\geq 0$, if $x\in T_{\ell'}$ then the radix expansion is
\begin{equation*}
x=0.c_{1}c_{2}c_{3}\ldots=\delta'\gamma'.c'_{1}c'_{2}c'_{3}\ldots.
\end{equation*}
It follows from Lemma \ref{neighbor generator} and $0=A^{-1}f(A)v$
that
\begin{equation*}
A^{-1}x+(b_{1}+k)A^{-1}v=0.(b_{1}+k)c_{1}c_{2}c_{3}\ldots=\delta\gamma.kc'_{1}c'_{2}c'_{3}\ldots\in
T_{\ell}
\end{equation*}
for $k=0,1,\dots,q-1-b_{1}$. The case when $b_{1}< 0$ can be proved
similarly.

For the second part,  we only need to show
\begin{equation*}
T_{\ell}\subset \bigcup_{\ell'\in B_{\ell}}\bigcup_{j\in
I_{b_{1}}}A^{-1}(T_{\ell'}+jv).
\end{equation*}
Let
$y=0.a_{1}a_{2}a_{3}\ldots=\delta\gamma.a_{1}'a_{2}'a_{3}'\ldots\in
T_{\ell}$.  It follows that $Ay-a_{1}v=0.a_{2}a_{3}a_{4}\ldots$
$=\delta\gamma(a_{1}'-a_{1}).a_{2}'a_{3}'\ldots\in T_{\ell'}$, where
$\ell'=A\ell-(a_{1}-a_{1}')v$.  This implies $y\in
A^{-1}(T_{\ell'}+a_{1}v)$. By definition, we see that $a_{1}\in
I_{b_{1}}$ for $b_{1}=a_{1}-a_{1}'$.
\end{proof}

It should be mentioned that the graph for the GIFS comes from the neighbor graph by adding more edges, or equivalently the neighbor graph is a reduced graph for the GIFS. The following example about Figure \ref{fig:0g} can illustrate their relationship. All the other cases are given in Appendix B.

\begin{Exa}\label{eg:gdset}
Consider the case $f(x)=x^{2}+px+q\;(p,q\geq
2,\;\text{excluding}\;p=q=2)$. When $\ell=v$, from Table
\ref{nbtb:x^2+px+q} we have $B_{\ell}=B_{v}=\{Av+pv,Av+(p-1)v\}$.
When $\ell'=Av+pv$, $b_{1}=-p$ and $I_{-p}=\{0,1,2,\dots,q-1-p\}$;
when $\ell'=Av+(p-1)v$, $b_{1}=-(p-1)$ and
$I_{-(p-1)}=\{0,1,2,\dots,q-p\}$.  Thus by Proposition
\ref{prop:union}, the first set equation comes out. Similarly the
other five can be deduced.  For simplicity we let $u_{1}=v$,
$u_{2}=Av+(p-1)v$, $u_{3}=Av+pv$.  Then the sets
$T_{\pm u_{1}},\ T_{\pm u_{2}},\ T_{\pm u_{3}}$, representing
$\partial T$ satisfy
\begin{eqnarray*}
AT_{u_{1}} &= &
\bigcup_{j=0}^{q-p} (T_{u_{2}}+jv)\cup\bigcup_{j=0}^{q-p-1} (T_{u_{3}}+jv)\\
AT_{u_{2}} &= &\bigcup_{j=0}^{p-2} (T_{-u_{2}}+jv)\cup\bigcup_{j=0}^{p-1} (T_{-u_{3}}+jv)\\
AT_{u_{3}} &= & T_{-u_{1}}\\
AT_{-u_{1}} &= &
\bigcup_{j=p-1}^{q-1} (T_{-u_{2}}+jv)\cup\bigcup_{j=p}^{q-1} (T_{-u_{3}}+jv)\\
AT_{-u_{2}} &= & \bigcup_{j=q-p+1}^{q-1} (T_{u_{2}}+jv)\cup\bigcup_{j=q-p}^{q-1} (T_{u_{3}}+jv)\\
AT_{-u_{3}} &= & T_{u_{1}}+(q-1)v\\
\end{eqnarray*}
\end{Exa}

The Hausdorff dimension ($\dim_H$) (see e.g., \cite{[Fa]}, \cite{[Fa2]}) is the most common and important dimension in fractal geometry. The case of self-similar sets has been studied extensively with or without separation conditions. However the case of self-affine sets is still hard to handle. Recently, He and Lau \cite{[HL]} defined the generalized Hausdorff dimension ($\dim_H^{\omega}$) and Hausdorff measure (${\mathcal H}_{\omega}^s$) for self-affine fractals by replacing the Euclidean norm with a pseudo-norm $\omega$ for which the expanding matrix $A$ becomes a similarity: $$\omega(Ax)=|\det{A}|^{1/2}\omega(x).$$
Under this setting, most of the basic properties for the self-similar sets can be carried to the self-affine sets. Moreover, Luo and Yang \cite{[LY]} extended this technique to the self-affine GIFS and obtained a dimension formula of the graph-directed set we need.

\begin{Prop}(\cite{[LY]}) \label{thm:LuoYang}
For the GIFS as in (\ref{eq:GDsets}) with the affine mappings $F_e(x)=A^{-1}(x+d_e)$ where $A$ is an expanding matrix and $|\det A|=|q|$, let $\rho(M)$ be the spectral radius of the contact matrix $M$. If the OSC holds, then $s=\dim_H^{\omega}E=2\log\rho(M)/\log{|q|}$ and $0<{\mathcal H}_{\omega}^s(E)<\infty$.
\end{Prop}

By using this, we can establish our first dimensional result about the boundary of $T$ as follows.

\begin{theorem}\label{thm:OSC}
The generalized Hausdorff dimension of the boundary of disk-like CC tile $T$ is $$\dim_H^{\omega}(\partial T)=2\log\rho(M)/\log{|q|}$$ and the corresponding measure is positive and finite.
\end{theorem}

\begin{proof}
From Propositions \ref{prop:union} and \ref{thm:LuoYang}, it suffices to show the GIFS representing the boundary of $T$ satisfies the OSC. Replacing $T_{\ell}$ by $(T+\ell)^\circ$, we can check the OSC holds case by case.  We illustrate the idea by proving the case $f(x)=x^{2}+px+q \ (p\geq 2,
q\geq 2,\ \text{excluding} \  p=q=2)$. In view of Example \ref{eg:gdset}, we need to show
\begin{eqnarray*}
A(T+u_{1})^{\circ} &\supset &
\bigcup_{j=0}^{q-p}\big((T+u_{2})^{\circ}+jv\big)
 \cup\bigcup_{j=0}^{q-p-1}\big((T+u_{3})^{\circ}+jv\big)\\
A(T+u_{2})^{\circ}  &\supset &
\bigcup_{j=0}^{p-2}\big((T-u_{2})^{\circ}+jv\big)
\cup\bigcup_{j=0}^{p-1}\big((T-u_{3})^{\circ}+jv\big)\\
A(T+u_{3})^{\circ} &\supset & (T-u_{1})^{\circ}\\
A(T-u_{1})^{\circ} &\supset & \bigcup_{j=p-1}^{q-1}\big((T-u_{2})^
{\circ}+jv\big)
 \cup\bigcup_{j=p}^{q-1} \big((T-u_{3})^{\circ}+jv\big)\\
A(T-u_{2}\big)^{\circ} &\supset & \bigcup_{j=q-p+1}^{q-1}
\big((T+u_{2})^{\circ}+jv\big)
 \cup\bigcup_{j=q-p}^{q-1} \big((T+u_{3})^{\circ}+jv\big)\\
A(T-u_{3})^{\circ} &\supset & (T+u_{1})^{\circ}+(q-1)v
\end{eqnarray*}
with disjoint unions. Since  $T$ is a CC tile, it follows that
\begin{equation}\label{eq:interior}
AT^{\circ}\supset\bigcup_{j=0}^{q-1}(T+jv)^{\circ}=\bigcup_{j=0}^{q-1}(T^{\circ}+jv)
\end{equation}
with disjoint union. By using  (\ref{eq:interior}) and
$0=f(A)v=A^{2}v+pAv+qv$ extensively, we prove the first two set
inequalities in the following. The remaining four can be verified
similarly.

For $j=0,1,\ldots,q-p$,
\begin{equation*}
\big(T+u_{2}\big)^{\circ}+jv = T^{\circ}+(p-1+j)v+Av \subset
A(T+u_{1})^{\circ}.
\end{equation*}

For $j=0,1,\ldots,q-p-1$,
\begin{equation*}
(T+u_{3})^{\circ}+jv = T^{\circ}+(p+j)v+Av \subset
A(T+u_{1})^{\circ}.
\end{equation*}

For $j=0,1,\ldots,p-2$,
\begin{eqnarray*}
\big(T-u_{2}\big)^{\circ}+jv &=& T^{\circ}+(j-p+1)v-Av\\
& =&  T^{\circ}+(q+j-p+1)v +A^{2}v+(p-1)Av\\
& \subset & A(T+u_{2})^{\circ}.
\end{eqnarray*}

For $j=0,1,\ldots,p-1$,
\begin{eqnarray*}
(T-u_{3})^{\circ}+jv &=& T^{\circ}+(j-p)v-Av\\
& = & T^{\circ}+(q+j-p)v+A^{2}v+(p-1)Av\\
& \subset &  A(T+u_{2})^{\circ}.
\end{eqnarray*}

By the same way, all the other cases follow and hence the theorem is
proved.
\end{proof}

In the rest  of this section, we will find the exact value of Hausdorff dimension $\dim_H (\partial T)$ for certain
particular cases that $A$ is a similarity. We state the simplest one first.

\begin{Prop}\label{prop:dim//gm}Let $A\in M_{2}(\ZZ)$ be expanding with characteristic
polynomial $f(x)=x^{2}+ q\ (|q|\geq 2)$ and $T(A,{\mathcal D})$  a
disk-like CC tile.  Then $\dim_H (\partial T)=1$.
\end{Prop}

\begin{proof}
By Theorem \ref{disk-likeness thm},  $T$ is a square tile (parallelogram). Hence $\dim_H (\partial T)=1$.
\end{proof}

Geometrically, a similarity is a multiple of either a reflection or
a rotation.  We call the former a {\it scaled reflection} and the
latter a {\it scaled rotation}; algebraically, a similarity is a
multiple of an orthogonal matrix. The case that $A$ is a scaled
reflection is solved already as its characteristic polynomial is of
the form $f(x)=x^{2}-q\ (q>0)$. So we focus our attention on those
$A$ that are scaled rotations.

\begin{Lem}\label{lem:rotate}
Let $A$ be a scaled rotation.  Then its characteristic polynomial
has positive constant term and $A$ has either two distinct non-real
eigenvalues or two equal real eigenvalues.

\end{Lem}
\begin{proof}
Let $A=\left(\begin{array}{cc}
r\mathrm{cos}\theta&-r\mathrm{sin}\theta\\
r\mathrm{sin}\theta&r\mathrm{cos}\theta
\end{array}\right)$.  The  characteristic polynomial is given
by $x^{2}-2r\mathrm{cos}\theta\,x+r^{2}$.  It has two equal real
zeros when $\theta=0\;\text{or}\;\pi$ and two distinct non-real
zeros otherwise.
\end{proof}

The following dimension formula on the boundaries of self-similar tiles has been investigated in the literature by various methods (see \cite{[DKV]},  \cite{[SW]}, \cite{[V]}, \cite{[HLR]}, \cite{[LN]}). We shall apply this formula to obtain our second dimensional result which is simpler than the known one.

\begin{Prop}\label{thm:dimformula}
If $A$ is a similarity with $|\det(A)|=|q|\geq 2$,  then the
Hausdorff dimension of $\partial T$ is given by
\begin{equation}\label{eq:hdim}
\dim_H (\partial T)=\log\rho(M)/\log r=2\log\rho(M)/\log |q|,
\end{equation}
where $\rho(M)$ denotes the spectral radius of the contact matrix
$M$ and $r=|q|^{1/2}$ is the expansion ratio of $A$.
\end{Prop}

Let $\ell,\ell',b_{1}$ and $B_{\ell}$ be defined as in Proposition~\ref{prop:union}. We first find the contact matrix $M$. Since ${\mathcal D}$ is a CC digit set, we have the entry $M_{\ell\ell'}=\#I_{b_{1}}=q-|b_{1}|$  where $I_{b_{1}}$ is as in Proposition \ref{prop:union}. Recall that $b_{1}$ is the  label of the edge from $\ell$ to $\ell'$. Hence we obtain the contact matrix $M$ of $T$ from its neighbor graph with different edge labels (i.e., replace $b_1$ by $q-|b_1|$).

Moreover, it is easy to see that there is a one-to-one correspondence between the contact matrix and the neighbor graph.  For example, the contact matrix for the case $f(x)=x^{2}+px+q\;(p,q\geq 2,\;2p\leq q+2\;\text{excluding}\;p=q=2)$ can be found in Table \ref{tab:x^2+px+q1}, and the related neighbor
graph is shown by  Figure \ref{fig:0g}. The contact matrices for the other cases are given in Appendix C.

\begin{table}
\centering{\small
\begin{tabular}{|c|c|c|c|c|c|c|c|c|}
\hline
&$v$&$Av+(p-1)v$&$Av+pv$&$-v$&$-Av-(p-1)v$&$-Av-pv$\\
\hline
$v$&$0$&$q-p+1$&$q-p$&$0$&$0$&$0$\\
$Av+(p-1)v$&$0$&$0$&$0$&$0$&$p-1$&$p$\\
$Av+pv$&$0$&$0$&$0$&$1$&$0$&$0$\\
$-v$&$0$&$0$&$0$&$0$&$q-p+1$&$q-p$\\
$-Av-(p-1)v$&$0$&$p-1$&$p$&$0$&$0$&$0$\\
$-Av-pv$&$1$&$0$&$0$&$0$&$0$&$0$\\
\hline
\end{tabular}}
\vspace{0.2cm} \caption{The contact matrix of $T$ associated with
$f(x)=x^{2}+px+q,\;p,q\geq 2,2p\leq q+2$\;(excluding\;$p=q=2$).}
\label{tab:x^2+px+q1}
\end{table}

If $M$ is irreducible (i.e., for each entry $M_{ij}$, there exists
an integer $n\geq 0$ such that $(M^n)_{ij}>0$), then the spectral
radius  $\rho(M)=\lambda_{M}$ where $\lambda_{M}$ is the {\it
Perron-Frobenius eigenvalue} of $M$ as stated in the following
simplified version of the Perron-Frobenius Theorem.

\begin{theorem}(\cite{[Ga]}, \cite{[Se]}) \label{thm:pf}Let $M$ be an irreducible non-negative matrix.  Then there exists a positive eigenvalue
$\lambda_{M}$ such that $\lambda_{M}\geq|\mu|$ for all eigenvalues
$\mu$ of $M$.  Moreover, $\lambda_{M}$ is a simple zero of the
characteristic polynomial of $M$.
\end{theorem}

It is known that a contact matrix is irreducible if and only if the
neighbor graph it represents is strongly connected. A directed graph
is called {\em strongly connected} if for any two vertices $v_i,
v_j$ there exists a path starting at $v_i$ and ending at $v_j$.

\begin{theorem}\label{thm:hsdf-dim}
Let $A\in M_{2}(\ZZ)$ be an expanding similarity with characteristic
polynomial $f(x)=x^{2}+px+q$ and $T(A,{\mathcal D})$ be a disk-like
CC tile. Then $\rho(M)$ is the largest positive zero of the cubic
polynomial
$$x^{3}-(|p|-1)x^{2}-(|q|-|p|)x-|q|.$$ Hence
$\dim_H \big(\partial T\big)=2\log \rho(M)/\log |q|$.
\end{theorem}

\begin{proof} Since $|\det(A)|=|q|$,  it is more convenient to work with
$f(x)=x^{2}\pm px \pm q\;(p\geq 0,\;q\geq 2)$. Also we ignore those
$f(x)$ of the form $f(x)=x^{2}\pm px - q\;(p>0,\; q\geq 2)$ as they
cannot be characteristic polynomials of similarities
(Lemma~\ref{lem:rotate}).  We can see from Appendix A or C  that the  contact matrix is irreducible if and only if $f(x)=x^2 \pm px+q$ where $p> 0$.

Case (1) $f(x)=x^{2}+px+q$. The characteristic polynomial of the
corresponding $M$ is $(x-1)(x^{2}+px+q)[x^{3}-(p-1)x^{2}-(q-p)x-q]$.
Notice that $\rho(M)\neq 1$. Indeed, if $\rho(M)=1$, then $\dim_{H}(\partial T)=0$, which implies $\partial T$ is totally disconnected (Proposition 2.5, \cite{[Fa]}). This is not possible for the boundary of a topological disk.  The zeros of $x^{2}+px+q$ are either both negative or both non-real. Hence $\rho(M)$ is the
largest positive real zero of $x^{3}-(p-1)x^{2}-(q-p)x-q$.

Case (2) $f(x)=x^{2}-px+q$.  The characteristic polynomial of the
corresponding $M$ is $(x+1)(x^{2}-px+q)[x^{3}-(p-1)x^{2}-(q-p)x-q]$.
Since $f(x)$ cannot have unequal real zeros (Lemma \ref{lem:rotate}), we have $p^{2}-4q\leq 0$. When $p^{2}-4q<0$, the zeros of $x^{2}-px+q$ are
non-real. Then $\rho(M)$ is the largest positive real zero of
$x^{3}-(p-1)x^{2}-(q-p)x-q$.  When $p^{2}-4q=0$, the two zeros of
$x^{2}-px+q$ are equal. But the Perron-Frobenius eigenvalue should
be a simple zero of the characteristic polynomial of $M$
(Theorem~\ref{thm:pf}), so $\rho(M)$ is also the largest positive
real zero of $x^{3}-(p-1)x^{2}-(q-p)x-q$.

Case (3) $f(x)=x^{2}+q$. The contact matrix $M$ is reducible. Its
characteristic polynomial is
$(x^{2}-q)(x^{2}+q)(x-1)(x+1)(x^{2}+1)$. We see that
$\rho(M)=q^{1/2}$, which is the largest positive zero of
$x^{3}+x^{2}-qx-q=(x^{2}-q)(x+1)$.

Case (4) $f(x)=x^{2}-q$.  The contact matrix $M$ is also reducible
and its characteristic polynomial is found to be
$(x^{2}-q)^{2}(x+1)(x-1)^{3}$. As in the previous case,
$\rho(M)=q^{1/2}$,  which is also the largest positive zero of
$x^{3}+x^{2}-qx-q$.
\end{proof}

\begin{Rem}
It is interesting to see that the signs of $p$ and $q$ do not matter
in the calculation of $\dim_H (\partial T)$ when $A$ is a
similarity. Notice also for the last two cases, $f(x)=x^{2}+
q\;(|q|\geq 2)$, we have $\rho(M)=|q|^{1/2}$.  It follows that
$\dim_H (\partial T)=1$, as expected for the boundary of a
parallelogram (Proposition~\ref{prop:dim//gm}).
\end{Rem}

\medskip

We observe that $\dim_H (\partial T)$ is independent of the choice
of the vector $v$ in the following sense.

\begin{Cor}\label{coro:indep-of-v}
Let $A\in M_{2}(\ZZ)$ be an expanding similarity with characteristic
polynomial $f(x)=x^{2}+px+q\;(|q|\geq 2)$. Let ${\mathcal
D}={\mathcal D}(v,|q|)$ and ${\mathcal D}'={\mathcal D}(v',|q|)$  be
two CC digit sets such that each of $\{v, Av\}$ and $\{v', Av'\}$ is
an independent set. If $2|p|\leq |q+2|$, then
\begin{equation*}
\dim_H \big(\partial T(A,{\mathcal D})\big)=\dim_H \big(\partial
T(A,{\mathcal D}')\big).
\end{equation*}
\end{Cor}

\begin{proof}
As $2|p|\leq |q+2|$, both $T(A,{\mathcal D})$ and $T(A,{\mathcal
D}')$ are disk-like CC tiles (Theorem \ref{disk-likeness thm}).
Hence the corollary follows from Theorem \ref{thm:hsdf-dim}.
\end{proof}

\begin{Rem}
We conjecture that Theorem \ref{thm:hsdf-dim} and Corollary
\ref{coro:indep-of-v} are also valid when $2|p|>|q+2|$, i.e.,  $T$
is non-disk-like.  The major difficulty in justifying these
conjectures is that, in general, there is no upper bound on the
number of neighbors of a non-disk-like CC tile \cite{[DJN]}.
\end{Rem}

\end{section}

\bigskip

\noindent {\it Acknowledgements:}  The authors would like to thank Professor Ka-Sing Lau for suggesting the question and reading an earlier version of the manuscript carefully. They are also grateful to the anonymous referees for their valuable comments and suggestions.

\bigskip


\begin{section}{\bf Appendix A: Neighbor Graphs}

Let $f(x)=x^{2}\pm px \pm q\;(p\geq 0,\;q\geq 2)$. The neighbor
graphs of disk-like tiles are classified by $f(x)$ and listed below.

\begin{figure}[htbp]
  \[
  \begin{xy}
    0;<1cm,0cm>: (0,0)*=<2cm,1cm>{v}*\frm{-}="a"
    ,(0,-2)*=<2cm,1cm>{Av-v}*\frm{-}="b"
    ,(0,-3.5)*=<2cm,1cm>{Av}*\frm{-}="c"
    ,(0,-5)*=<2cm,1cm>{Av+v}*\frm{-}="x"
    ,(6,0)*=<2cm,1cm>{-v}*\frm{-}="d"
    ,(6,-2)*=<2cm,1cm>{-Av+v}*\frm{-}="e"
    ,(6,-3.5)*=<2cm,1cm>{-Av}*\frm{-}="f"
    ,(6,-5)*=<2cm,1cm>{-Av-v}*\frm{-}="y"
    ,"a"+(0,-0.5)
    ;"b"+(0,0.5)**\dir{-}?*^!/10pt/{1}?>*\dir{>}
    ,"b"+(1,0)
    ;"y"+(-1,0)**\crv{(3,-4)&(3,-6)}?(0.3)*^!/20pt/{-(q-1)}?>*\dir{>}
    ,"y"+(1,0.5)
    ;"e"+(1,-0.5)**\crv{(8,-5)}?(0.7)*^!/25pt/{q-1}?>*\dir{>}
    ,"e"+(-1,0)
    ;"x"+(1,0)**\crv{(3,-4)&(3,-6)}?(0.3)*_!/20pt/{q-1}?>*\dir{>}
    ,"x"+(-1,0.5)
    ;"b"+(-1,-0.5)**\crv{(-2,-5)}?(0.7)*_!/25pt/{-(q-1)}?>*\dir{>}
    ,"a"+(-1,-0.5)
    ;"c"+(-1,0)**\crv{(-2,-3)}?*^!/10pt/{0}?>*\dir{>}
    ,"c"+(1,0)
    ;"d"+(-0.7,-0.5)**\crv{(4,-4)}?(0.9)*_!/25pt/{-(q-1)}?>*\dir{>}
    ,"d"+(0,-0.5)
    ;"e"+(0,0.5)**\dir{-}?*^!/10pt/{-1}?>*\dir{>}
    ,"f"+(-1,0)
    ;"a"+(0.7,-0.5)**\crv{(2,-4)}?(0.9)*^!/25pt/{q-1}?>*\dir{>}
    ,"d"+(1,-0.5)
    ;"f"+(1,0)**\crv{(8,-3)}?*_!/10pt/{0}?>*\dir{>}
    ,"a"+(-1,0)
    ;"x"+(-1,0)**\crv{(-6,-6)}?*^!/10pt/{-1}?>*\dir{>}
    ,"d"+(1,0)
    ;"y"+(1,0)**\crv{(12,-6)}?*_!/10pt/{1}?>*\dir{>}
  \end{xy}
  \]
  \caption{The neighbor graph of $T$ associated with $f(x)=x^2+q$.}
  \label{fig:1g}
\end{figure}

\medskip

\begin{figure}[htbp]
  \[
  \begin{xy}
    0;<1cm,0cm>: (0,0)*=<2cm,1cm>{v}*\frm{-}="a"
    ,(0,-2)*=<2cm,1cm>{Av-v}*\frm{-}="b"
    ,(0,-3.5)*=<2cm,1cm>{Av}*\frm{-}="c"
    ,(0,-5)*=<2cm,1cm>{Av+v}*\frm{-}="x"
    ,(4,0)*=<2cm,1cm>{-v}*\frm{-}="d"
    ,(4,-2)*=<2cm,1cm>{-Av+v}*\frm{-}="e"
    ,(4,-3.5)*=<2cm,1cm>{-Av}*\frm{-}="f"
    ,(4,-5)*=<2cm,1cm>{-Av-v}*\frm{-}="y"
    ,"a"+(0,-0.5)
    ;"b"+(0,0.5)**\dir{-}?*^!/10pt/{1}?>*\dir{>}
    ,"a"+(-1,0)
    ;"c"+(-1,0)**\crv{(-5,-4)}?(0.2)*^!/8pt/{0}?>*\dir{>}
    ,"c"+(-1,0.5)
    ;"a"+(-1,-0.5)**\crv{(-1.5,-3)}?(0.4)*_!/25pt/{q-1}?>*\dir{>}
    ,"a"+(-1,0.5)
    ;"x"+(-1,0)**\crv{(-6,-2)}?(0.2)*^!/10pt/{-1}?>*\dir{>}
    ,"b"+(1,0.3)
    ;"e"+(-1,0.3)**\dir{-}?*_!/10pt/{q-1}?>*\dir{>}
    ,"e"+(-1,-0.3)
    ;"b"+(1,-0.3)**\dir{-}?*_!/10pt/{-(q-1)}?>*\dir{>}
    ,"d"+(0,-0.5)
    ;"e"+(0,0.5)**\dir{-}?*_!/10pt/{-1}?>*\dir{>}
    ,"d"+(1,0)
    ;"f"+(1,0)**\crv{(9,-4)}?(0.2)*_!/8pt/{0}?>*\dir{>}
    ,"f"+(1,0.5)
    ;"d"+(1,-0.5)**\crv{(5.5,-3)}?(0.4)*^!/25pt/{-(q-1)}?>*\dir{>}
    ,"d"+(1,0.5)
    ;"y"+(1,0)**\crv{(10,-2)}?(0.2)*_!/10pt/{1}?>*\dir{>}
    ,"x"+(-0.5,-0.5)
    ;"x"+(0.5,-0.5)**\crv{(-1,-6)&(1,-6)}?*^!/10pt/{q-1}?>*\dir{>}
    ,"y"+(0.5,-0.5)
    ;"y"+(-0.5,-0.5)**\crv{(5,-6)&(3,-6)}?*_!/10pt/{-(q-1)}?>*\dir{>}
  \end{xy}
  \]
  \caption{The neighbor graph of $T$ associated with $f(x)=x^2-q$.}
  \label{fig:2g}
\end{figure}

\clearpage

\begin{figure}[htbp]
  \[
  \begin{xy}
    0;<1cm,0cm>: (0,0)*=<2cm,1cm>{v}*\frm{-}="a"
    ,(0,-2)*=<3cm,1cm>{Av}*\frm{-}="b"
    ,(0,-4)*=<2cm,1cm>{Av+v}*\frm{-}="c"
    ,(6,0)*=<2cm,1cm>{-v}*\frm{-}="d"
    ,(6,-2)*=<3cm,1cm>{-Av}*\frm{-}="e"
    ,(6,-4)*=<2cm,1cm>{-Av-v}*\frm{-}="f"
    ,"a"+(0,-0.5)
    ;"b"+(0,0.5)**\dir{-}?*^!/10pt/{0}?>*\dir{>}
    ,"b"+(1.5,0)
    ;"f"+(-1,0.5)**\dir{-}?(0.7)*^!/15pt/{-(q-1)}?>*\dir{>}
    ,"c"+(1,-0.5)
    ;"d"+(1,0.3)**\crv{"f"+(0,-2)&"f"+(4,0)&"d"+(4,0)}?(0.3)*^!/15pt/{-(q-1)}?>*\dir{>}
    ,"a"+(-1,0)
    ;"c"+(-1,0)**\crv{(-4,-3)}?*^!/10pt/{-1}?>*\dir{>}
    ,"d"+(0,-0.5)
    ;"e"+(0,0.5)**\dir{-}?*_!/10pt/{0}?>*\dir{>}
    ,"e"+(-1.5,0)
    ;"c"+(1,0.5)**\dir{-}?(0.8)*_!/10pt/{q-1}?>*\dir{>}
    ,"f"+(-1,-0.5)
    ;"a"+(-1,0.3)**\crv{"c"+(0,-2)&"c"+(-4,0)&"a"+(-4,0)}?(0.3)*_!/15pt/{q-1}?>*\dir{>}
    ,"d"+(1,0)
    ;"f"+(1,0)**\crv{(10,-3)}?*_!/10pt/{1}?>*\dir{>}
  \end{xy}
  \]
  \caption{The neighbor graph of $T$ associated with $f(x)=x^2+x+q$.}
  \label{fig:3g}
\end{figure}

\medskip

\begin{figure}[htbp]
  \[
  \begin{xy}
    0;<1cm,0cm>: (0,0)*=<2cm,1cm>{v}*\frm{-}="a"
    ,(0,-2)*=<3cm,1cm>{Av}*\frm{-}="b"
    ,(0,-4)*=<2cm,1cm>{Av-v}*\frm{-}="c"
    ,(6,0)*=<2cm,1cm>{-v}*\frm{-}="d"
    ,(6,-2)*=<3cm,1cm>{-Av}*\frm{-}="e"
    ,(6,-4)*=<2cm,1cm>{-Av+v}*\frm{-}="f"
    ,"a"+(0,-0.5)
    ;"b"+(0,0.5)**\dir{-}?*^!/10pt/{0}?>*\dir{>}
    ,"b"+(0,-0.5)
    ;"c"+(0,0.5)**\dir{-}?*^!/25pt/{-(q-1)}?>*\dir{>}
    ,"c"+(1,0.5)
    ;"d"+(-1,-0.5)**\dir{-}?(0.2)*^!/15pt/{-(q-1)}?>*\dir{>}
    ,"a"+(-1,0)
    ;"c"+(-1,0)**\crv{(-4,-3)}?*^!/10pt/{-1}?>*\dir{>}
    ,"d"+(0,-0.5)
    ;"e"+(0,0.5)**\dir{-}?*_!/10pt/{0}?>*\dir{>}
    ,"e"+(0,-0.5)
    ;"f"+(0,0.5)**\dir{-}?*_!/25pt/{q-1}?>*\dir{>}
    ,"f"+(-1,0.5)
    ;"a"+(1,-0.5)**\dir{-}?(0.2)*_!/15pt/{q-1}?>*\dir{>}
    ,"d"+(1,0)
    ;"f"+(1,0)**\crv{(10,-3)}?*_!/10pt/{1}?>*\dir{>}
  \end{xy}
  \]
  \caption{The neighbor graph of $T$ associated with $f(x)=x^2-x+q$.}
  \label{fig:4g}
\end{figure}

\medskip

\begin{figure}[htbp]
  \[
  \begin{xy}
    0;<1cm,0cm>: (0,0)*=<2cm,1cm>{v}*\frm{-}="a"
    ,(0,-3)*=<3cm,1cm>{Av+(p-1)v}*\frm{-}="b"
    ,(0,-6)*=<2cm,1cm>{Av+pv}*\frm{-}="c"
    ,(6,0)*=<2cm,1cm>{-v}*\frm{-}="d"
    ,(6,-3)*=<3cm,1cm>{-Av-(p-1)v}*\frm{-}="e"
    ,(6,-6)*=<2cm,1cm>{-Av-pv}*\frm{-}="f"
    ,"a"+(0,-0.5)
    ;"b"+(0,0.5)**\dir{-}?*^!/25pt/{-(p-1)}?>*\dir{>}
    ,"b"+(1.5,-0.5)
    ;"f"+(-1,0.5)**\dir{-}?(0.1)*^!/20pt/{-(q-p)}?>*\dir{>}
    ,"e"+(-1.5,-0.5)
    ;"c"+(1,0.5)**\dir{-}?(0.1)*_!/20pt/{q-p}?>*\dir{>}
    ,"b"+(1.5,0.3)
    ;"e"+(-1.5,0.3)**\dir{-}?*_!/10pt/{-(q-p+1)}?>*\dir{>}
    ,"e"+(-1.5,-0.3)
    ;"b"+(1.5,-0.3)**\dir{-}?*_!/10pt/{q-p+1}?>*\dir{>}
    ,"a"+(-1,0)
    ;"c"+(-1,0)**\crv{(-4,-3)}?*^!/10pt/{-p}?>*\dir{>}
    ,"d"+(0,-0.5)
    ;"e"+(0,0.5)**\dir{-}?*_!/25pt/{p-1}?>*\dir{>}
    ,"d"+(1,0)
    ;"f"+(1,0)**\crv{(10,-3)}?*_!/10pt/{p}?>*\dir{>}
    ,"c"+(1,-0.5)
    ;"d"+(1,0.3)**\crv{"f"+(0,-2)&"f"+(4,0)&"d"+(4,0)}?(0.3)*^!/15pt/{-(q-1)}?>*\dir{>}
    ,"f"+(-1,-0.5)
    ;"a"+(-1,0.3)**\crv{"c"+(0,-2)&"c"+(-4,0)&"a"+(-4,0)}?(0.3)*_!/15pt/{q-1}?>*\dir{>}
  \end{xy}
  \]
  \caption{The neighbor graph of $T$ associated with $f(x)=x^2+px+q$, $p\geq 2$, $2p\leq q+2$\;
  (excluding $q=p=2$).}
  \label{fig:5g}
\end{figure}

\medskip

\begin{figure}[htbp]
  \[
  \begin{xy}
    0;<1cm,0cm>: (0,0)*=<2cm,1cm>{v}*\frm{-}="a"
    ,(0,-3)*=<3cm,1cm>{Av-(p-1)v}*\frm{-}="b"
    ,(0,-6)*=<2cm,1cm>{Av-pv}*\frm{-}="c"
    ,(6,0)*=<2cm,1cm>{-v}*\frm{-}="d"
    ,(6,-3)*=<3cm,1cm>{-Av+(p-1)v}*\frm{-}="e"
    ,(6,-6)*=<2cm,1cm>{-Av+pv}*\frm{-}="f"
    ,"a"+(0,-0.5)
    ;"b"+(0,0.5)**\dir{-}?*^!/25pt/{p-1}?>*\dir{>}
    ,"b"+(1.5,-0.5)
    ;"b"+(1.5,0.5)**\crv{(3,-6)&(3,0)}?(0.8)*^!/10pt/{-(q-p+1)}?>*\dir{>}
    ,"b"+(0,-0.5)
    ;"c"+(0,0.5)**\dir{-}?*^!/25pt/{-(q-p)}?>*\dir{>}
    ,"a"+(-1,0)
    ;"c"+(-1,0)**\crv{(-4,-3)}?*^!/10pt/{p}?>*\dir{>}
    ,"d"+(0,-0.5)
    ;"e"+(0,0.5)**\dir{-}?*_!/25pt/{-(p-1)}?>*\dir{>}
    ,"e"+(-1.5,-0.5)
    ;"e"+(-1.5,0.5)**\crv{(3,-6)&(3,0)}?(0.2)*_!/10pt/{q-p+1}?>*\dir{>}
    ,"e"+(0,-0.5)
    ;"f"+(0,0.5)**\dir{-}?*_!/25pt/{q-p}?>*\dir{>}
    ,"d"+(1,0)
    ;"f"+(1,0)**\crv{(10,-3)}?*_!/10pt/{-p}?>*\dir{>}
    ,"c"+(1,-0.5)
    ;"d"+(1,0.3)**\crv{"f"+(0,-2)&"f"+(4,0)&"d"+(4,0)}?(0.3)*^!/15pt/{-(q-1)}?>*\dir{>}
    ,"f"+(-1,-0.5)
    ;"a"+(-1,0.3)**\crv{"c"+(0,-2)&"c"+(-4,0)&"a"+(-4,0)}?(0.3)*_!/15pt/{q-1}?>*\dir{>}
  \end{xy}
  \]
  \caption{The neighbor graph of $T$ associated with $f(x)=x^2-px+q$, $p\geq 2$, $2p\leq q+2$\;
   (excluding $q=p=2$)}
  \label{fig:6g}
\end{figure}

\medskip

\begin{figure}[htbp]
  \[
  \begin{xy}
    0;<1cm,0cm>: (0,0)*=<2cm,1cm>{v}*\frm{-}="a"
    ,(0,-2)*=<3cm,1cm>{Av+pv}*\frm{-}="b"
    ,(0,-4)*=<3cm,1cm>{Av+(p+1)v}*\frm{-}="c"
    ,(6,0)*=<2cm,1cm>{-v}*\frm{-}="d"
    ,(6,-2)*=<3cm,1cm>{-Av-pv}*\frm{-}="e"
    ,(6,-4)*=<3cm,1cm>{-Av-(p+1)v}*\frm{-}="f"
    ,"a"+(-0.5,-0.5)
    ;"b"+(-0.5,0.5)**\dir{-}?*^!/10pt/{-p}?>*\dir{>}
    ,"b"+(0.5,0.5)
    ;"a"+(0.5,-0.5)**\dir{-}?*^!/25pt/{q-1}?>*\dir{>}
    ,"f"+(0,0.5)
    ;"e"+(0,-0.5)**\dir{-}?*^!/25pt/{-(q-p)}?>*\dir{>}
    ,"c"+(0,0.5)
    ;"b"+(0,-0.5)**\dir{-}?*_!/15pt/{q-p}?>*\dir{>}
    ,"a"+(-1,0)
    ;"c"+(-1.5,0)**\crv{(-4,-3)}?(0.2)*^!/25pt/{-(p+1)}?>*\dir{>}
    ,"d"+(0.5,-0.5)
    ;"e"+(0.5,0.5)**\dir{-}?*_!/10pt/{p}?>*\dir{>}
    ,"e"+(-0.5,0.5)
    ;"d"+(-0.5,-0.5)**\dir{-}?*_!/25pt/{-(q-1)}?>*\dir{>}
    ,"c"+(-0.5,-0.5)
    ;"c"+(0.5,-0.5)**\crv{(-1.5,-5.5)&(1.5,-5.5)}?*^!/10pt/{q-p-1}?>*\dir{>}
    ,"f"+(0.5,-0.5)
    ;"f"+(-0.5,-0.5)**\crv{(7.5,-5.5)&(4.5,-5.5)}?*_!/10pt/{-(q-p-1)}?>*\dir{>}
    ,"d"+(1,0)
    ;"f"+(1.5,0)**\crv{(10,-3)}?(0.2)*_!/25pt/{p+1}?>*\dir{>}
  \end{xy}
  \]
  \caption{The neighbor graph of $T$ associated with $f(x)=x^2+px-q$, $p\geq 1$, $2p\leq q-2$}
  \label{fig:7g}
\end{figure}

\medskip

\begin{figure}[htbp]
  \[
  \begin{xy}
    0;<1cm,0cm>: (0,0)*=<2cm,1cm>{v}*\frm{-}="a"
    ,(0,-3)*=<3cm,1cm>{Av-pv}*\frm{-}="b"
    ,(0,-6)*=<3cm,1cm>{Av-(p+1)v}*\frm{-}="c"
    ,(6,0)*=<2cm,1cm>{-v}*\frm{-}="d"
    ,(6,-3)*=<3cm,1cm>{-Av+pv}*\frm{-}="e"
    ,(6,-6)*=<3cm,1cm>{-Av+(p+1)v}*\frm{-}="f"
    ,"a"+(-0.5,-0.5)
    ;"b"+(-0.5,0.5)**\dir{-}?*^!/10pt/{p}?>*\dir{>}
    ,"b"+(0.5,0.5)
    ;"a"+(0.5,-0.5)**\dir{-}?*^!/25pt/{q-1}?>*\dir{>}
    ,"f"+(-1,0.5)
    ;"b"+(1.5,-0.5)**\dir{-}?(0.1)*^!/15pt/{-(q-p)}?>*\dir{>}
    ,"c"+(1,0.5)
    ;"e"+(-1.5,-0.5)**\dir{-}?(0.1)*_!/15pt/{q-p}?>*\dir{>}
    ,"a"+(-1,0)
    ;"c"+(-1.5,0)**\crv{(-4,-3)}?(0.2)*^!/25pt/{p+1}?>*\dir{>}
    ,"d"+(0.5,-0.5)
    ;"e"+(0.5,0.5)**\dir{-}?*_!/10pt/{-p}?>*\dir{>}
    ,"e"+(-0.5,0.5)
    ;"d"+(-0.5,-0.5)**\dir{-}?*_!/25pt/{-(q-1)}?>*\dir{>}
    ,"c"+(1.5,0)
    ;"f"+(-1.5,0)**\dir{-}?*_!/10pt/{q-p-1}?>*\dir{>}
    ,"f"+(-1.5,-0.3)
    ;"c"+(1.5,-0.3)**\dir{-}?*_!/10pt/{-(q-p-1)}?>*\dir{>}
    ,"d"+(1,0)
    ;"f"+(1.5,0)**\crv{(10,-3)}?(0.2)*_!/25pt/{-(p+1)}?>*\dir{>}
  \end{xy}
  \]
  \caption{The neighbor graph of $T$ associated with $f(x)=x^2-px-q$, $p\geq 1$, $2p\leq q-2$}
  \label{fig:8g}
\end{figure}

\medskip

\begin{figure}[htbp]
  \[
  \begin{xy}
    0;<1cm,0cm>: (0,0)*=<2cm,1cm>{v}*\frm{-}="a"
    ,(0,-1.5)*=<2cm,1cm>{Av+v}*\frm{-}="b"
    ,(0,-3)*=<2cm,1cm>{Av+2v}*\frm{-}="c"
    ,(6,0)*=<2cm,1cm>{-v}*\frm{-}="d"
    ,(6,-1.5)*=<2cm,1cm>{-Av-v}*\frm{-}="e"
    ,(6,-3)*=<2cm,1cm>{-Av-2v}*\frm{-}="f"
    ,"a"+(0,-0.5)
    ;"b"+(0,0.5)**\dir{-}?*^!/10pt/{-1}?>*\dir{>}
    ,"b"+(1,0.3)
    ;"e"+(-1,0.3)**\dir{-}?*_!/10pt/{-1}?>*\dir{>}
    ,"b"+(0,-0.5)
    ;"f"+(-1,0)**\dir{-}?(0.3)*^!/8pt/{0}?>*\dir{>}
    ,"d"+(0,-0.5)
    ;"e"+(0,0.5)**\dir{-}?*_!/8pt/{1}?>*\dir{>}
    ,"e"+(-1,-0.3)
    ;"b"+(1,-0.3)**\dir{-}?*_!/10pt/{1}?>*\dir{>}
    ,"e"+(0,-0.5)
    ;"c"+(1,0)**\dir{-}?(0.3)*_!/8pt/{0}?>*\dir{>}
    ,"f"+(1,0)
    ;"a"+(0,0.5)**\crv{"f"+(3,0)&"d"+(2.5,0.25)&"d"+(1.5,1.75)&"a"+(0,1.75)}?*^!/10pt/{1}?>*\dir{>}
    ,"c"+(-1,0)
    ;"d"+(0,0.5)**\crv{"c"+(-3,0)&"a"+(-2.5,0.25)&"a"+(-1.5,1.75)&"d"+(0,1.75)}?*_!/10pt/{-1}?>*\dir{>}
  \end{xy}
  \]
  \caption{The neighbor graph of $T$ associated with $f(x)=x^2+2x+2$}
  \label{fig:9g}
\end{figure}

\medskip

\begin{figure}[htbp]
  \[
  \begin{xy}
    0;<1cm,0cm>: (0,0)*=<2cm,1cm>{v}*\frm{-}="a"
    ,(0,-3)*=<2cm,1cm>{Av-v}*\frm{-}="b"
    ,(0,-7)*=<2cm,1cm>{Av-2v}*\frm{-}="c"
    ,(6,0)*=<2cm,1cm>{-v}*\frm{-}="d"
    ,(6,-3)*=<2cm,1cm>{-Av+v}*\frm{-}="e"
    ,(6,-7)*=<2cm,1cm>{-Av+2v}*\frm{-}="f"
    ,"a"+(0,-0.5)
    ;"b"+(0,0.5)**\dir{-}?*^!/10pt/{1}?>*\dir{>}
    ,"b"+(-1,-0.5)
    ;"b"+(-1,0.5)**\crv{(-3,-6)&(-3,0)}?(0.8)*_!/10pt/{-1}?>*\dir{>}
    ,"b"+(0,-0.5)
    ;"c"+(0,0.5)**\dir{-}?*^!/10pt/{0}?>*\dir{>}
    ,"c"+(1,0.5)
    ;"d"+(-1,-0.5)**\dir{-}?(0.3)*_!/10pt/{-1}?>*\dir{>}
    ,"d"+(0,-0.5)
    ;"e"+(0,0.5)**\dir{-}?*_!/10pt/{-1}?>*\dir{>}
    ,"e"+(0,-0.5)
    ;"f"+(0,0.5)**\dir{-}?*_!/10pt/{0}?>*\dir{>}
    ,"e"+(1,-0.5)
    ;"e"+(1,0.5)**\crv{(9,-6)&(9,0)}?(0.8)*^!/10pt/{1}?>*\dir{>}
    ,"f"+(-1,0.5)
    ;"a"+(1,-0.5)**\dir{-}?(0.3)*^!/10pt/{1}?>*\dir{>}
  \end{xy}
  \]
  \caption{The neighbor graph of $T$ associated with $f(x)=x^2-2x+2$}
  \label{fig:10g}
\end{figure}
\end{section}

\clearpage
\newpage

\begin{section}{\bf Appendix B: Graph-directed Sets}

Let $f(x)=x^{2}\pm px \pm q\;(p\geq 0,\;q\geq 2)$. The
graph-directed sets representing the boundary $\partial T$ are
classified by $f(x)$ and listed below.

\medskip

(1) $f(x)=x^{2}+q$. Convention:
$u_{1}=v$,\;$u_{2}=Av-v$,\;$u_{3}=Av$,\;$u_{4}=Av+v$. {\scriptsize
\begin{eqnarray*}
AT_{u_{1}} &= &
\bigcup_{j=1}^{q-1}(T_{u_{2}}+jv)\cup\bigcup_{j=0}^{q-1}(T_{u_{3}}+jv)
\cup\bigcup_{j=0}^{q-2}(T_{u_{4}}+jv)\\
AT_{u_{2}} &= &  T_{-u_{4}}\\
AT_{u_{3}} &= & T_{-u_{1}}\\
AT_{u_{4}} &= & T_{u_{2}}\\
AT_{-u_{1}} &= &
\bigcup_{j=0}^{q-2}(T_{-u_{2}}+jv)\cup\bigcup_{j=0}^{q-1}(T_{-u_{3}}+jv)
\cup\bigcup_{j=1}^{q-1}(T_{-u_{4}}+jv)\\
AT_{-u_{2}} &= & T_{u_{4}}+(q-1)v\\
AT_{-u_{3}} &= & T_{u_{1}}+(q-1)v\\
AT_{-u_{4}} &= & T_{-u_{2}}+(q-1)v
\end{eqnarray*}}

(2) $f(x)=x^{2}-q$.  Convention:
$u_{1}=v$,\;$u_{2}=Av-v$,\;$u_{3}=Av$,\;$u_{4}=Av+v$. {\scriptsize
\begin{eqnarray*}
AT_{u_{1}} &= &
\bigcup_{j=1}^{q-1}(T_{u_{2}}+jv)\cup\bigcup_{j=0}^{q-1}(T_{u_{3}}+jv)\cup\bigcup_{j=0}^{q-2}(T_{u_{4}}+jv)\\
AT_{u_{2}} &= & T_{-u_{2}}+(q-1)v\\
AT_{u_{3}} &= & T_{u_{1}}+(q-1)v\\
AT_{u_{4}} &= & T_{u_{4}}+(q-1)v\\
AT_{-u_{1}} &= &
\bigcup_{j=0}^{q-2}(T_{-u_{2}}+jv)\cup\bigcup_{j=0}^{q-1}(T_{-u_{3}}+jv)\cup\bigcup_{j=1}^{q-1}(T_{-u_{4}}+jv)\\
AT_{-u_{2}} &= & T_{u_{2}}\\
AT_{-u_{3}} &=& T_{-u_{1}}\\
AT_{-u_{4}} &= & T_{-u_{4}}\\
\end{eqnarray*}}

(3) $f(x)=x^{2}+x+q$. Convention:
$u_{1}=v$,\;$u_{2}=Av$,\;$u_{3}=Av+v$. {\scriptsize
\begin{eqnarray*}
AT_{u_{1}} &= &
\bigcup_{j=0}^{q-1}(T_{u_{2}}+jv)\cup\bigcup_{j=0}^{q-2}(T_{u_{3}}+jv)\\
AT_{u_{2}} &= & T_{-u_{3}}\\
AT_{u_{3}} &= & T_{-u_{1}}\\
AT_{-u_{1}} &= &
\bigcup_{j=0}^{q-1}(T_{-u_{2}}+jv)\cup\bigcup_{j=1}^{q-1}(T_{-u_{3}}+jv)\\
AT_{-u_{2}} &= & T_{u_{3}}+(q-1)v\\
AT_{-u_{3}} &= & T_{u_{1}}+(q-1)v\\
\end{eqnarray*}}

(4) $f(x)=x^{2}-x+q$.  Convention:
$u_{1}=v$,\;$u_{2}=Av$,\;$u_{3}=Av-v$. {\scriptsize
\begin{eqnarray*}
AT_{u_{1}} &= &
\bigcup_{j=0}^{q-1}(T_{u_{2}}+jv)\cup\bigcup_{j=0}^{q-2}(T_{u_{3}}+jv)\\
AT_{u_{2}} &= & T_{u_{3}}\\
AT_{u_{3}} &= & T_{-u_{1}}\\
AT_{-u_{1}} &= &
\bigcup_{j=0}^{q-1}(T_{-u_{2}}+jv)\cup\bigcup_{j=1}^{q-1}(T_{-u_{3}}+jv)\\
AT_{-u_{2}}  &= & T_{-u_{3}}+(q-1)v\\
AT_{-u_{3}} &= & T_{u_{1}}+(q-1)v\\
\end{eqnarray*}}

(5) $f(x)=x^{2}+px+q,\;p\geq 2$,\;$2p\leq q+2$\;(excluding $p=q=2$).
Convention:$u_{1}=v$,\;$u_{2}=Av+(p-1)v$,\;$u_{3}=Av+pv$.
{\scriptsize
\begin{eqnarray*}
AT_{u_{1}} &= &
\bigcup_{j=0}^{q-p}(T_{u_{2}}+jv)\cup\bigcup_{j=0}^{q-p-1}(T_{u_{3}}+jv)\\
AT_{u_{2}} &= &\bigcup_{j=0}^{p-2}(T_{-u_{2}}+jv)\cup\bigcup_{j=0}^{p-1}(T_{-u_{3}}+jv)\\
AT_{u_{3}} &= & T_{-u_{1}}\\
AT_{-u_{1}} &= &
\bigcup_{j=p-1}^{q-1}(T_{-u_{2}}+jv)\cup\bigcup_{j=p}^{q-1}(T_{-u_{3}}+jv)\\
AT_{-u_{2}} &= & \bigcup_{j=q-p+1}^{q-1}(T_{u_{2}}+jv)\cup\bigcup_{j=q-p}^{q-1}(T_{u_{3}}+jv)\\
AT_{-u_{3}} &= & T_{u_{1}}+(q-1)v\\
\end{eqnarray*}}

(6) $f(x)=x^{2}-px+q,\;p\geq 2$,\;$2p\leq
q+2$\;(excluding\;$p=q=2$).  Convention:
$u_{1}=v$,\;$u_{2}=Av-(p-1)v$,\;$u_{3}=Av-pv$. {\scriptsize
\begin{eqnarray*}
AT_{u_{1}} &= &
\bigcup_{j=p-1}^{q-1}(T_{u_{2}}+jv)\cup\bigcup_{j=p}^{q-1}(T_{u_{3}}+jv)\\
AT_{u_{2}} &= &\bigcup_{j=0}^{p}(T_{u_{2}}+jv)\cup\bigcup_{j=0}^{p-1}(T_{u_{3}}+jv)\\
AT_{u_{3}} &= & T_{-u_{1}}\\
AT_{-u_{1}} &= &
\bigcup_{j=0}^{q-p}(T_{-u_{2}}+jv)\cup\bigcup_{j=0}^{q-p-1}(T_{-u_{3}}+jv)\\
AT_{-u_{2}} &= & \bigcup_{j=q-p+1}^{q-1}(T_{-u_{2}}+jv)\cup\bigcup_{j=q-p}^{q-1}(T_{-u_{3}}+jv)\\
AT_{-u_{3}} &= & T_{u_{1}}+(q-1)v\\
\end{eqnarray*}}

(7) $f(x)=x^{2}+px-q,\;p\geq 1$,\;$2p\leq q-2$.  Convention:
$u_{1}=v$,\;$u_{2}=Av+pv$,\;$u_{3}=Av+(p+1)v$. {\scriptsize
\begin{eqnarray*}
AT_{u_{1}} &= &
\bigcup_{j=0}^{q-p-1}(T_{u_{2}}+jv)\cup\bigcup_{j=0}^{q-p-2}(T_{u_{3}}+jv)\\
AT_{u_{2}} &= & T_{u_{1}}+(q-1)v\\
AT_{u_{3}} &= & \bigcup_{j=q-p}^{q-1}(T_{u_{2}}+jv)\cup\bigcup_{j=q-p-1}^{q-1}(T_{u_{3}}+jv)\\
AT_{-u_{1}} & = &
\bigcup_{j=p}^{q-1}(T_{-u_{2}}+jv)\cup\bigcup_{j=p+1}^{q-1}(T_{-u_{3}}+jv)\\
AT_{-u_{2}} &= &  T_{u_{1}}\\
AT_{-u_{3}} &= &\bigcup_{j=0}^{p-1}(T_{-u_{2}}+jv)\cup\bigcup_{j=0}^{p}(T_{-u_{3}}+jv)\\
\end{eqnarray*}}

(8) $f(x)=x^{2}-px-q,\;p\geq 1$,\;$2p\leq q-2$.  Convention:
$u_{1}=v$,\;$u_{2}=Av-pv$,\;$u_{3}=Av-(p+1)v$. {\scriptsize
\begin{eqnarray*}
AT_{u_{1}} &= &
\bigcup_{j=p}^{q-1}(T_{u_{2}}+jv)\cup\bigcup_{j=p+1}^{q-1}(T_{u_{3}}+jv)\\
AT_{u_{2}} &= & T_{u_{1}}+(q-1)v\\
AT_{u_{3}} &= &\bigcup_{j=q-p}^{q-1}(T_{-u_{2}}+jv)\cup\bigcup_{j=q-p-1}^{q-1}(T_{-u_{3}}+jv)\\
AT_{-u_{1}} &= &
\bigcup_{j=0}^{q-p-1}(T_{-u_{2}}+jv)\cup\bigcup_{j=0}^{q-p-2}(T_{-u_{3}}+jv)\\
AT_{-u_{2}} &= &  T_{-u_{1}}\\
AT_{-u_{3}} &= &\bigcup_{j=0}^{p-1}(T_{u_{2}}+jv)\cup\bigcup_{j=0}^{p-2}(T_{u_{3}}+jv)\\
\end{eqnarray*}}

(9) $f(x)=x^{2}+2x+2$. Convention:
$u_{1}=v$,\;$u_{2}=Av+v$,\;$u_{3}=Av+2v$. {\scriptsize
\begin{eqnarray*}
AT_{u_{1}} &= & T_{u_{2}}\\
AT_{u_{2}} &= & T_{-u_{2}}\cup T_{-u_{3}}\cup (T_{-u_{3}}+v)\\
AT_{u_{3}} &= & T_{-u_{1}}\\
AT_{-u_{1}} &= & T_{-u_{2}}+v\\
AT_{-u_{2}} &= & (T_{u_{2}}+v)\cup T_{u_{3}}\cup (T_{u_{3}}+v)\\
AT_{-u_{3}} &= & T_{u_{1}}+v\\
\end{eqnarray*}}

 (10) $f(x)=x^{2}-2x+2$. Convention: $u_{1}=v$,\;$u_{2}=Av-v$,\;$u_{3}=Av-2v$.
{\scriptsize
\begin{eqnarray*}
AT_{u_{1}}  &= & T_{u_{2}}+v\\
AT_{u_{2}} &= & T_{u_{2}}\cup T_{u_{3}}\cup (T_{u_{3}}+v)\\
AT_{u_{3}} &= & T_{-u_{1}}\\
AT_{-u_{1}} &= & T_{-u_{2}}\\
AT_{-u_{2}} &= & (T_{-u_{2}}+v)\cup T_{-u_{3}}\cup (T_{-u_{3}}+v)\\
AT_{-u_{3}} &= & T_{u_{1}}+v\\
\end{eqnarray*}}

\end{section}

\begin{section}{\bf Appendix C: Contact Matrices}

Let $f(x)=x^{2}\pm px \pm q\;(p\geq 0,\;q\geq 2)$.  The contact
matrices (in table form) of disk-like tiles are classified by $f(x)$
and listed below.

{\tiny
\begin{table}[htbp]
\centering
\begin{tabular}{|c|c|c|c|c|c|c|c|c|}
\hline
&$v$&$Av$&$-v$&$-Av$&$Av-v$&$-Av-v$&$-Av+v$&$Av+v$\\
\hline
$v$&$0$&$q$&$0$&$0$&$q-1$&$0$&$0$&$q-1$\\
$Av$&$0$&$0$&$1$&$0$&$0$&$0$&$0$&$0$\\
$-v$&$0$&$0$&$0$&$q$&$0$&$q-1$&$q-1$&$0$\\
$-Av$&$1$&$0$&$0$&$0$&$0$&$0$&$0$&$0$\\
$Av-v$&$0$&$0$&$0$&$0$&$0$&$1$&$0$&$0$\\
$-Av-v$&$0$&$0$&$0$&$0$&$0$&$0$&$1$&$0$\\
$-Av+v$&$0$&$0$&$0$&$0$&$0$&$0$&$0$&$1$\\
$Av+v$&$0$&$0$&$0$&$0$&$1$&$0$&$0$&$0$\\
\hline
\end{tabular}
\vspace{0.2cm} \caption{$f(x)=x^{2}+q$.} \label{tab:x^2+q}
\end{table}}

{\tiny
\begin{table}[htbp]
\centering
\begin{tabular}{|c|c|c|c|c|c|c|c|c|}
\hline
&$v$&$Av$&$-v$&$-Av$&$Av-v$&$-Av+v$&$Av+v$&$-Av-v$\\
\hline
$v$&$0$&$q$&$0$&$0$&$q-1$&$0$&$q-1$&$0$\\
$Av$&$1$&$0$&$0$&$0$&$0$&$0$&$0$&$0$\\
$-v$&$0$&$0$&$0$&$q$&$0$&$q-1$&$0$&$q-1$\\
$-Av$&$0$&$0$&$1$&$0$&$0$&$0$&$0$&$0$\\
$Av-v$&$0$&$0$&$0$&$0$&$0$&$1$&$0$&$0$\\
$-Av+v$&$0$&$0$&$0$&$0$&$1$&$0$&$0$&$0$\\
$Av+v$&$0$&$0$&$0$&$0$&$0$&$0$&$1$&$0$\\
$-Av-v$&$0$&$0$&$0$&$0$&$0$&$0$&$0$&$1$\\
\hline
\end{tabular}
\vspace{0.2cm} \caption{$f(x)=x^{2}-q$.} \label{tab:x^2-q}
\end{table}}

{\tiny
\begin{table}[htbp]
\centering
\begin{tabular}{|c|c|c|c|c|c|c|c|c|}
\hline
&$v$&$Av$&$Av+v$&$-v$&$-Av$&$-Av-v$\\
\hline
$v$&$0$&$q$&$q-1$&$0$&$0$&$0$\\
$Av$&$0$&$0$&$0$&$0$&$0$&$1$\\
$Av+v$&$0$&$0$&$0$&$1$&$0$&$0$\\
$-v$&$0$&$0$&$0$&$0$&$q$&$q-1$\\
$-Av$&$0$&$0$&$1$&$0$&$0$&$0$\\
$-Av-v$&$1$&$0$&$0$&$0$&$0$&$0$\\
\hline
\end{tabular}
\vspace{0.2cm} \caption{$f(x)=x^{2}+x+q$.} \label{tab:x^2+x+q}
\end{table}}

{\tiny
\begin{table}[htbp]
\centering
\begin{tabular}{|c|c|c|c|c|c|c|c|c|}
\hline
&$v$&$Av$&$Av-v$&$-v$&$-Av$&$-Av+v$\\
\hline
$v$&$0$&$q$&$q-1$&$0$&$0$&$0$\\
$Av$&$0$&$0$&$1$&$0$&$0$&$0$\\
$Av-v$&$0$&$0$&$0$&$1$&$0$&$0$\\
$-v$&$0$&$0$&$0$&$0$&$q$&$q-1$\\
$-Av$&$0$&$0$&$0$&$0$&$0$&$1$\\
$-Av+v$&$1$&$0$&$0$&$0$&$0$&$0$\\
\hline
\end{tabular}
\vspace{0.2cm} \caption{$f(x)=x^{2}-x+q$.} \label{tab:x^2-x+q}
\end{table}}

{\tiny
\begin{table}[htbp]
\centering
\begin{tabular}{|c|c|c|c|c|c|c|c|c|}
\hline
&$v$&$Av+(p-1)v$&$Av+pv$&$-v$&$-Av-(p-1)v$&$-Av-pv$\\
\hline
$v$&$0$&$q-p+1$&$q-p$&$0$&$0$&$0$\\
$Av+(p-1)v$&$0$&$0$&$0$&$0$&$p-1$&$p$\\
$Av+pv$&$0$&$0$&$0$&$1$&$0$&$0$\\
$-v$&$0$&$0$&$0$&$0$&$q-p+1$&$q-p$\\
$-Av-(p-1)v$&$0$&$p-1$&$p$&$0$&$0$&$0$\\
$-Av-pv$&$1$&$0$&$0$&$0$&$0$&$0$\\
\hline
\end{tabular}
\vspace{0.2cm} \caption{$f(x)=x^{2}+px+q,\;p\geq 2,\;2p\leq
q+2$\;(excluding\;$p=q=2$).} \label{tab:x^2+px+q}
\end{table}}

{\tiny
\begin{table}[htbp]
\centering
\begin{tabular}{|c|c|c|c|c|c|c|c|c|}
\hline
&$v$&$Av-(p-1)v$&$Av-pv$&$-v$&$-Av+(p-1)v$&$-Av+pv$\\
\hline
$v$&$0$&$q-p+1$&$q-p$&$0$&$0$&$0$\\
$Av-(p-1)v$&$0$&$p-1$&$p$&$0$&$0$&$0$\\
$Av-pv$&$0$&$0$&$0$&$1$&$0$&$0$\\
$-v$&$0$&$0$&$0$&$0$&$q-p+1$&$q-p$\\
$-Av+(p-1)v$&$0$&$0$&$0$&$0$&$p-1$&$p$\\
$-Av+pv$&$1$&$0$&$0$&$0$&$0$&$0$\\
\hline
\end{tabular}
\vspace{0.2cm} \caption{$f(x)=x^{2}-px+q,\;p\geq 2,\;2p\leq
q+2$\;(excluding\;$p=q=2$).} \label{tab:x^2-px+q}
\end{table}}

{\tiny
\begin{table}[htbp]
\centering
\begin{tabular}{|c|c|c|c|c|c|c|c|c|}
\hline
&$v$&$Av+pv$&$Av+(p+1)v$&$-v$&$-Av-pv$&$-Av-(p+1)v$\\
\hline
$v$&$0$&$q-p$&$q-p-1$&$0$&$0$&$0$\\
$Av+pv$&$1$&$0$&$0$&$0$&$0$&$0$\\
$Av+(p+1)v$&$0$&$p$&$p+1$&$0$&$0$&$0$\\
$-v$&$0$&$0$&$0$&$0$&$q-p$&$q-p-1$\\
$-Av-pv$&$0$&$0$&$0$&$1$&$0$&$0$\\
$-Av-(p+1)v$&$0$&$0$&$0$&$0$&$p$&$p+1$\\
\hline
\end{tabular}
\vspace{0.2cm} \caption{$f(x)=x^{2}+px-q,\;p\geq 1,\;2p\leq q-2$.}
\label{tab:x^2+px-q}
\end{table}}

{\tiny
\begin{table}[htbp]
\centering
\begin{tabular}{|c|c|c|c|c|c|c|c|c|}
\hline
&$v$&$Av-pv$&$Av-(p+1)v$&$-v$&$-Av+pv$&$-Av+(p+1)v$\\
\hline
$v$&$0$&$q-p$&$q-p-1$&$0$&$0$&$0$\\
$Av-pv$&$1$&$0$&$0$&$0$&$0$&$0$\\
$Av-(p+1)v$&$0$&$0$&$0$&$0$&$p$&$p+1$\\
$-v$&$0$&$0$&$0$&$0$&$q-p$&$q-p-1$\\
$-Av+pv$&$0$&$0$&$0$&$1$&$0$&$0$\\
$-Av+(p+1)v$&$0$&$p$&$p+1$&$0$&$0$&$0$\\
\hline
\end{tabular}
\vspace{0.2cm} \caption{$f(x)=x^{2}-px-q,\;p\geq 1,\;2p\leq q-2$.}
\label{tab:x^2-px-q}
\end{table}}

\clearpage
\newpage

{\tiny
\begin{table}[htbp]
\centering
\begin{tabular}{|c|c|c|c|c|c|c|c|c|}
\hline
&$v$&$Av+v$&$Av+2v$&$-v$&$-Av-v$&$-Av-2v$\\
\hline
$v$&$0$&$1$&$0$&$0$&$0$&$0$\\
$Av+v$&$0$&$0$&$0$&$0$&$1$&$2$\\
$Av+2v$&$0$&$0$&$0$&$1$&$0$&$0$\\
$-v$&$0$&$0$&$0$&$0$&$1$&$0$\\
$-Av-v$&$0$&$1$&$2$&$0$&$0$&$0$\\
$-Av-2v$&$1$&$0$&$0$&$0$&$0$&$0$\\
\hline
\end{tabular}
\vspace{0.2cm} \caption{$f(x)=x^{2}+2x+2$.} \label{tab:x^2+2x+2}
\end{table}}

{\tiny
\begin{table}[htbp]
\centering
\begin{tabular}{|c|c|c|c|c|c|c|c|c|}
\hline
&$v$&$Av-v$&$Av-2v$&$-v$&$-Av+v$&$-Av+2v$\\
\hline
$v$&$0$&$1$&$0$&$0$&$0$&$0$\\
$Av-v$&$0$&$1$&$2$&$0$&$0$&$0$\\
$Av-2v$&$0$&$0$&$0$&$1$&$0$&$0$\\
$-v$&$0$&$0$&$0$&$0$&$1$&$0$\\
$-Av+v$&$0$&$0$&$0$&$0$&$1$&$2$\\
$-Av+2v$&$1$&$0$&$0$&$0$&$0$&$0$\\
\hline
\end{tabular}
\vspace{0.2cm} \caption{$f(x)=x^{2}-2x+2$.} \label{tab:x^2-2x+2}
\end{table}
}

\end{section}

\end{document}